\documentclass[11pt,oneside]{article}
\usepackage[utf8]{inputenc}
\usepackage[english]{babel}
\usepackage{amsmath}
\usepackage{amsfonts}
\usepackage{amssymb}
\usepackage{amsthm}
\usepackage{tikz-cd}  
\usepackage[left=3cm,right=3cm,top=3cm,bottom=3cm]{geometry}

\newcommand\blfootnote[1]{%
  \begingroup
  \renewcommand\thefootnote{}\footnote{#1}%
  \addtocounter{footnote}{-1}%
  \endgroup
}
\newtheorem{thm}{Theorem}[section]

\newtheorem{prop}[thm]{Proposition}
\newtheorem{cor}[thm]{Corollary}
\newtheorem{rem}[thm]{Remark}
\newtheorem{lem}[thm]{Lemma}
\newtheorem{ex}[thm]{Example}
\newtheorem{df}[thm]{Definition}

\newtheorem{thmx}{Theorem}

\author{Pedro J. Chocano, Manuel. A. Morón, Francisco R. Ruiz del Portal}
\title{Coincidence theorems for finite topological spaces}
\date{}
\begin{document}
\maketitle
\begin{abstract}

We adapt the definition of the Vietoris map to the framework of finite topological spaces and we prove some coincidence theorems. From them, we deduce a Lefschetz fixed point theorem for multivalued maps that improves recent results in the literature. Finally, it is given an application to the approximation of discrete dynamical systems in polyhedra.
\end{abstract}

\section{Introduction, preliminaries and motivation}

\blfootnote{2010  Mathematics  Subject  Classification: 55M20, 37B99, 54H20, 06A0}
\blfootnote{Key words and phrases: Finite $T_0$ spaces, Alexandroff spaces, posets, multivalued maps, fixed points, dynamical systems, approximation of polyhedra.}
\blfootnote{This research is partially supported by Grants PGC2018-098321-B-I00 and BES-2016-076669 from Ministerio de Ciencia, Innovación y Universidades (Spain).}

Finite topological spaces are becoming a significant part of topology and a good tool so as to model and face problems of different nature. For instance, they can be used to reconstruct compact metric spaces \cite{mondejar2020reconstruction} or to solve realization problems of groups in topological categories \cite{chocano2020topological}. From an algebraic point of view, they are interesting since they have the same homotopy and singular homology groups of simplicial complexes \cite{mccord1966singular}. There are two monographs that treat precisely the alegbraic aspects of finite topological spaces, \cite{barmak2011algebraic,may1966finite}. Recently, an interest in finding applications of finite spaces to dynamical systems has grown up. These spaces, due to its finitude, seems a good candidate to develop computational methods. In \cite{lipinski2020conley}, L. Lipi\`nski, J. Kubica, M. Mrozek and T. Wanner define an analogue of continuous dynamical systems for finite spaces and construct a Conley index. To do that, they generalize the theory of combinatorial multivector fields for Lefschetz complexes introduced by M. Mrozek in \cite{mrozek2017conley}. On the other hand, the theory of combinatorial multivector fields is also a generalization of the concept of combinatorial vector field introduced by R. Forman in \cite{forma1998morse,forma1998combinatorial}. Combinatorial vector fields are used to adapt the classical Morse theory in a combinatorial way for CW-complexes. This theory has found several applications. For example, in \cite{mischaikow2013morse}, K. Mischaikow and V. Nanda, using discrete Morse theory, develop a method to compute persistent homology more efficiently

In the same spirit of adapting the theory of dynamical systems to finite spaces, J.A. Barmak, M. Mrozek and T. Wanner, in \cite{barmak2020Lefschetz}, obtained a Lefschetz fixed point theorem for a special class of multivalued maps. A Lefschetz fixed point theorem for finite spaces and continuous maps was obtained by K. Baclawski and A. Björner in \cite{baclawski1979fixed}, where they also relate the Lefschetz number with the Euler characteristic of the set of fixed points. The idea of this paper is to adapt the classical notion of Vietoris map and Vietoris-Begle mapping theorem to finite topological spaces  so as to obtain a coincidence theorem, which is a generalization of the Lefschetz fixed point theorem. 
From there, we deduce Lefschtez fixed point theorems and we also extend the class of multivalued maps for which the result holds true, i.e., we generalize the Lefschetz fixed point theorem obtained in \cite{barmak2020Lefschetz} to a more flexible class of multivalued maps. 
These ideas are combined with the reconstruction of compact polyhedra by finite topological spaces in order to give some results about approximation of dynamical systems.  

We recall basic definitions and results from the literature before enunciating the main results of this paper. Firstly, we review important concepts in the theory of Alexandroff spaces. An Alexandroff space $X$ is a topological space with the property that the arbitrary intersection of open sets is open. If $X$ is an Alexandroff space and $x\in X$, then $U_x$ denotes the open set that is defined as the intersection of all open sets containing $x$. Analogously, $F_x$ denotes the closed set that is given by the intersection of all closed sets containing $x$. Given a partially ordered set (poset) $(X,\leq)$, a lower (upper) set $S$ is a subset of $X$ such that if $x\in S$ and $y\leq x$ ($x\leq y$), then $y\in S $. In addition, if $x,y\in X$, we denote $x\prec y$ ($x\succ y$) if and only if $x<y$ ($x>y$) and there is no $z\in X$ with $x<z<y$ ($x>z>y$). In \cite{alexandroff1937diskrete}, P.S. Alexandroff proved that for a poset $(X,\leq)$, the family of lower (upper) sets of $\leq$ is a $T_0$ topology on $X$, that makes $X$ an Alexanfroff space and for a $T_0$ Alexandroff space, the relation $x\leq_{\tau} y$ if and only if $U_x\subset U_y$ ($U_y \subset U_x$) is a partial order on $X$. Moreover, the two associations relating $T_0$ topologies and partial orders are mutually inverse.  The topology generated by the upper sets is usually called the opposite topology. The partial order considered in parenthesis is usually called the opposite order. If $X$ is an Alexandroff space, where we are not considering the opposite order, and $x\in X$, $U_x$ can also be seen as the set $\{y\in X|y\leq x\}$. Similarly, $F_x$ can also be seen as the set $\{y\in X|y\geq x\}$. Furthermore, $U_x$ and $F_x$ are contractible spaces. 

Every finite $T_0$ topological space is an Alexandroff space. Therefore, we will treat finite $T_0$ topological spaces and finite partially ordered sets as the same object without explicit mention. In this context, some topological notions such as continuity or homotopy can be expressed in terms of partial orders. For instance, a map $f:X\rightarrow Y$ between finite $T_0$ topological spaces is continuous if and only if it is order-preserving. If $f,g:X\rightarrow Y$ are two continuous maps between finite $T_0$ topological spaces, $f$ is homotopic to $g$ if and only if there exists a finite sequence of continuous maps $f_1,...,f_n$ such that $f=f_1\leq f_2 \geq ...\leq f_n=g $.  See \cite{barmak2011algebraic,may1966finite} for a complete exposition.

Moreover, weak homotopy equivalences play a central role in the theory of Alexandroff spaces.
\begin{df} A weak homotopy equivalence is a map between topological spaces which induces isomorphisms on all homotopy groups. Let $X$ and $Y$ be topological spaces, $X$ is weak homotopy equivalent to $Y$ if there exists a sequence of spaces $X=X_0,X_1,...,X_n=Y$ such that there are weak homotopy equivalences $X_i\rightarrow X_{i+1}$ or $X_{i+1}\rightarrow X_i$ for every $0\leq i\leq n-1$.
\end{df}
\begin{rem} Weak homotopy equivalences satisfy the $2$-out-$3$ property, that is to say, let $f:X\rightarrow Y$ and $g:Y\rightarrow Z$ be maps, if $2$ of the $3$ maps $f,g,g\circ f$ are weak homotopy equivalences, then so is the third.
\end{rem}
We recall some results of \cite{mccord1966singular}.
\begin{df} If $X$ is an Alexandroff space, $\mathcal{K}(X)$ denotes the McCord complex, which is a simplicial complex whose simplices are totally finite ordered subsets of $X$. Given a  simplicial complex $L$, $\mathcal{X}(L)$ denotes the face poset, i.e., the poset whose points are the simplices of $L$, the partial order is given by the subset relation. $|L|$ denotes the geometric realization.
\end{df}
\begin{rem}\label{rem:subdivsion} If $L$ is a simplicial complex, then it is easy to check that $\mathcal{K}(\mathcal{X}(L))$ is the barycentric subdivision of $L$.
\end{rem}
Given a simplicial complex $K$, when there is no confusion we will also denote by $K$ the geometric realization of the simplicial complex. If $f:K\rightarrow L$ is a simplicial map between two simplicial complexes, we denote by $|f|$ the extension of $f$ to $|K|$. Again, we will also use $f$ to denote $|f|$.

\begin{thm}\label{thm:McCord1}\cite[Theorem 2]{mccord1966singular} There exists a correspondence that assigns to each $T_0$ Alexandroff space a simplicial complex $\mathcal{K}(X)$ and a weak homotopy equivalence $f_X:|\mathcal{K}(X)|\rightarrow X$. Each map $\varphi:X\rightarrow Y$ of $T_0$ Alexandroff spaces is also a simplicial map $\mathcal{K}(\varphi):\mathcal{K}(X)\rightarrow \mathcal{K}(Y)$, and $\varphi\circ f_X=f_Y\circ \mathcal{K}(\varphi)$.
\end{thm}
The weak homotopy equivalence $f_X:|\mathcal{K}(X)|\rightarrow X$ considered in Theorem \ref{thm:McCord1} is given as follows, for each $u\in |\mathcal{K}(X)|$ we have that $u$ is contained in a unique open simplex $\sigma_u$, where $\sigma_u$ is given by a chain $v_0<...<v_n$ of $X$, then $f_X(u)=v_0$. If $\varphi:X\rightarrow Y$ is a continuous map between two finite $T_0$ topological spaces, then $\varphi$ sends chains to chains. Hence, $\mathcal{K}(\varphi)(v_0<...<v_n)=\varphi(v_0)<...<\varphi(v_n)$.
\begin{thm}\cite[Theorem 3]{mccord1966singular}\label{thm:McCordX} There exists a correspondence that assigns to each simplicial complex $K$ a $T_0$ Alexandroff space $\mathcal{X}(K)$ and a weak homotopy equivalence $f_K:|K|\rightarrow \mathcal{X}(K)$. Furthermore, to each simplicial map $\psi:K\rightarrow L$ is assigned a map $\mathcal{X}(\psi):\mathcal{X}(K)\rightarrow \mathcal{X}(L)$ such that $\mathcal{X}(\psi)\circ f_K$ is homotopic to $f_L\circ |\psi|$.
\end{thm}
Given a simplicial complex $L$, $L_i$ denotes the set of $i$-simplices of $L$. Finally, we recall a definition that will appear in subsequent sections.
\begin{df} Let $f,g:|K|\rightarrow |L|$ be maps and $K,L$ be  simplicial complexes. $f$ is simplicially close to $g$ if and only if for every $x$, there exists a simplex $\sigma_x$ in $L$ containing in its closure $f(x)$ and $g(x)$.
\end{df}

The Euler characteristic of a finite poset $(X,\leq)$ is defined as the alternate sum of the number of $i$-chains, where an $i$-chain in $X$ consists of $i+1$ elements in $X$ such that $v_0<...<v_i$. The Euler characteristic of $X$ is denoted by $\chi(X)$. Then, 
$$\chi(X)=\sum_{i=0}(-1)^i |\{ v_0<...<v_i|v_i\in X \ \text{for every} \ i\}|,$$
where $|.|$ denotes the cardinal of a set.

\begin{rem} It is trivial to check that for a finite $T_0$ topological space $X$, the Euler characteristic of $X$ defined previously coincides with the classical Euler characteristic defined for the simplicial complex $\mathcal{K}(X)$.
\end{rem}

Given a finite $T_0$ topological space $X$. The Hasse diagram of $X$ is a directed graph. The vertices are the points of $X$ and there is an edge between two points $x$ and $y$ if and only if $x\prec y$. The direction of the edge goes from the lower element to the upper element. In subsequent Hasse diagrams we omit the orientation of the edges and we assume an upward orientation.

A dynamical system for a topological space $X$ consists of a triad $(\mathbb{T},X,\varphi)$, where $\mathbb{T}$ is usually $\mathbb{Z}$ (discrete dynamical system) or $\mathbb{R}$ (continuous dynamical system) and $\varphi:\mathbb{T}\times X\rightarrow X$ is a continuous function satisfying the following: $\varphi(0,x)=x$ for every $x\in X$ and $\varphi(t+s,x)=\varphi(t,\varphi(s,x))$ for all $s,t\in \mathbb{T}$ and $x\in X$. If $t\in \mathbb{T}$ is fixed, we denote $\varphi_t:X\rightarrow X$ for simplicity. The classical definition of dynamical system for an Alexandroff space seems uninteresting. 

\begin{prop} If $A$ is an Alexandroff space, the only continuous dynamical system $\varphi: \mathbb{R} \times A \rightarrow A$  is the trivial one, i.e., $\varphi_t:A\rightarrow A$ is the identity map for every $t\in \mathbb{R}$.
\begin{proof}
We will treat $A$ as a poset $(A,\leq)$ with the opposite order, that is to say, we consider the upper sets. Then, if $x\in A$, $F_x$ is an open set. We argue by contradiction. Suppose that $\varphi_t$ is not the identity map for every $t\in \mathbb{R}$. Then, there exists $s\in \mathbb{R}$ with $\varphi_s(x)\neq x$ for some $x\in X$. On the other hand, $\varphi$ is continuous at $(0,x)$, so there exists $\epsilon>0$ such that $\varphi( (-\epsilon, \epsilon)\times F_x )\subseteq F_x$. Concretely, it can be deduced that for every $t\in (-\epsilon,\epsilon)$, $\varphi_t:F_x\rightarrow F_x$ is a homeomorphism. We prove that $\varphi_t(F_x)=F_x$. Suppose there exists $z\in F_x\setminus \varphi_t(F_x)$. Since $-t\in (-\epsilon,\epsilon)$, we have that $\varphi_{-t}(z)\in F_x$, which implies the contradiction. From here, it is immediate to deduced the desired assertion. It is not complicate to prove that $s$ can be considered in $(-\epsilon,\epsilon)$. $\varphi_s:F_x\rightarrow F_x$ is a homeomorphism. Thus, there exists $z\in F_x$ such that $\varphi_s(z)=x$. $\varphi_s$ should preserve the order, so $x=\varphi_s(z)> \varphi_s(x)=y$, but $y>x$. 

\end{proof}
\end{prop}

Furthermore, if $X$ is a finite $T_0$ topological space and we have a discrete dynamical system $\varphi: \mathbb{Z}\times X\rightarrow X$, we get that $\varphi_1=f$ is a homeomorphism. Hence, it is easy to check that there exists a natural number $n$ such that $f^n=id$, where $id$ denotes the identity map.

By the previous arguments, it seems natural to consider multivalued maps so as to be able to establish new definitions of dynamical systems in this context. A first approach to start can be to develop a proper fixed point theory. In \cite{barmak2020Lefschetz}, J.A. Barmak, M. Mrozek and T. Wanner provide a Lefschetz fixed point theorem for multivalued maps and finite $T_0$ topological spaces. Before recalling the definitions and results obtained in \cite{barmak2020Lefschetz}, we recall the classical Lefschetz fixed point theorem. 

Given a finite polyhedron $X$ and a continuous map $f:X\rightarrow X$. The Lefschetz number of $f$ is defined as follows:
$$\Lambda(f)=\sum_{i=0}(-1)^i tr(f_*:H_i(X)\rightarrow H_i(X))$$ 
where $f_*$ denotes the linear map induced by $f$ on the torsion-free part of the homology groups of $X$ and $tr$ denotes the trace. This definition can be also extended to general topological spaces that have homology groups finitely generated. $Fix(f)$ denotes the topological subspace of $X$ given by the fixed points of $f$. Moreover, if $f_*$ is a linear map on the torsion-free part of the homology groups of a topological space $X$ that is not induced by a continuous map, we will denote also by $\Lambda(f_*)$ the alternate sum of the traces of $f_*$. 
\begin{thm}[Lefschetz fixed point theorem]\label{thm:lefschtezfixedpointTheoremPoliedros} Given a finite polyhedron $X$ and a continuous function $f:X\rightarrow X$. If $\Lambda(f)\neq 0$, there exists a point $x\in X$ such that $x=f(x)$.

\end{thm}

A proof of this theorem can be found for instance in \cite{hatcher2000algebraic} or \cite{munkres1984elements}. Moreover, a finite version of this theorem can be found in \cite{baclawski1979fixed}. K. Baclawski and A. Björner proved the following result, where the Euler characteristic that appears is the one defined previously for finite $T_0$ topological spaces (finite posets).
\begin{thm}\cite[Theorem 1.1]{baclawski1979fixed} Let $P$ be a finite poset and let $f:P\rightarrow P$ be an order-preserving map. Then $\Lambda(f)$ is the Euler characteristic of $Fix(f)$. In particular, if $\Lambda(f)\neq 0$, then $Fix(f)\neq \emptyset$.
\end{thm}
Hence, the Lefschetz fixed point theorem for finite $T_0$ topological spaces provides more information about the structure of the fixed points set of a continuous map.

\begin{df} A topological space $X$ is acyclic if the homology groups of $X$ are isomorphic to the homology groups of a point. 
\end{df}
If $X$ is an acyclic finite polyhedron and $f:X\rightarrow X$ is a continuous map, $f$ has a fixed point. Then, the Lefschetz fixed point theorem can be seen as a generalization of the Brouwer fixed point theorem.

A generalization of the Lefschetz fixed point theorem is the so-called coincidence theorem, it can be found in \cite{eilenberg1946fixed} proved by S. Eilenberg and D. Montgomery. Given two continuous maps $f,g:X\rightarrow Y$, it is said that $f$ and $g$ have a coincidence point if there exists $x\in X$ such that $f(x)=g(x)$. Before enunciating the coincidence theorem, we recall a result of L. Vietoris \cite{vietoris1927uber} that was generalized by E. G. Begle \cite{begle1950vietoris}. For simplicity, the results will be only enunciated for compact polyhedra.
\begin{df}\label{def:Vietorismapclasico} A continuous map $f:X\rightarrow Y$, where $X$ and $Y$ are two compact polyhedra, is a Vietoris map if $f^{-1}(y)$ is acyclic for every $y\in Y$.
\end{df}
\begin{thm}[Vietoris-Begle mapping theorem]\label{thm:vietorisClasico}
If $f:X\rightarrow Y$ is a Vietoris map, where $X$ and $Y$ are compact polyhedra, then $f$ induces isomorphism in the homology groups. 
\end{thm}
\begin{thm}[Coincidence theorem]\label{thm:coincidenciaClasico} Let $X$ and $Y$ be compact polyhedra. If $f,g:X\rightarrow Y$ are continuous maps, where $g$ is a Vietoris map, then $\Lambda(f_*\circ g^{-1}_*)$ is defined, and if $\Lambda(f_*\circ g^{-1}_*)\neq 0$, there exists a point $x\in X$ such that $f(x)=g(x)$.
\end{thm}

We recall some definitions and results of \cite{barmak2020Lefschetz}. We will take as definitions some characterizations instead of the original ones in order to simplify.

\begin{df}\label{def:semicontinuous}\cite[Lemma 3.2]{barmak2020Lefschetz} Let $F:X\multimap Y$ be an arbitrary multivalued map between two finite $T_0$ topological spaces. It is said that $F$ is upper semicontinuous (lower semicontinuous) if for all $x_1,x_2\in X$ with $x_1\leq x_2$ ($x_1\geq x_2$) and for all $y_1\in F(x_1)$ there exists $y_2\in F(x_2)$ such that $y_1\leq y_2$ ($y_1\geq y_2$).
\end{df} 
The notion of upper semicontinuity is defined for general topological spaces. Concretely, a multivalued map $F:X\multimap Y$ between two topological spaces is an upper semicontinuous multivalued map if for each point $x\in X$ and for each neighborhood $V$ of $F(x)$ in $Y$ there exists a neighborhood $U$ of $x$ in $X$ such that $F(U)=\bigcup_{x\in X}F(x)$ is contained in $V$. Hence, Definition \ref{def:semicontinuous} is a particularization of the previous one to finite topological spaces.

\begin{df}\label{def:strongsemicontinuous}\cite[Lemma 3.4]{barmak2020Lefschetz} Let $F:X\multimap Y$ denote an arbitrary multivalued map between two finite $T_0$ topological spaces. It is said that $F$ is strongly upper semicontinuous (strongly lower semicontinuous) or susc (slsc) if for all $x_1,x_2\in X$ with $x_1\leq x_2$ ($x_1\geq x_2$), $F(x_1)\subseteq F(x_2)$ ($F(x_1)\supseteq F(x_2)$).
\end{df} 

From now on, if $F:X\multimap Y$ is a multivalued map between two finite $T_0$ topological spaces, we will denote by $\Gamma(F)$ the graph of $F$, that is to say, $\Gamma(F)=\{(x,y)|y\in F(x)\}$. Hence, $\Gamma(F)$ can also be seen as a finite $T_0$ topological space, where we are taking the lexicographic order on $\Gamma(F)\subseteq X\times Y$. We denote by $p:\Gamma(F)\rightarrow X$ and $q:\Gamma(F)\rightarrow Y$ the projections of the first and the second coordinate of $\Gamma(F)$, respectively. If $F:X\multimap X$ is a multivalued map, $x\in X$ is called fixed point if $x\in F(x)$.

In \cite[Lemma 4.1]{barmak2020Lefschetz}, it is shown that if $F:X\multimap Y$ is a susc (slsc) multivalued map with acyclic values between two finite $T_0$ topological spaces, then $p$ induces isomorphism in homology. Therefore, it makes sense to consider $F_*:H_n(X)\rightarrow H_n(Y)$ defined as $F_*=q_*\circ p_*^{-1}$. From here, it is proved a Lefschetz fixed point theorem.
\begin{thm}\cite[Theorem 5.3]{barmak2020Lefschetz}\label{thm:lefschetzBarmakMrozek} Let $X$ be an arbitrary finite $T_0$ topological space and let $F:X\multimap X$ be a susc or slsc mutlivalued map with acyclic values. If $\Lambda(F)\neq 0$, then $F$ has a fixed point.
\end{thm}

Now, we enunciate three of the main results of the paper. These results relies on the notion of Vietoris-like maps and multivalued maps, which plays a central role herein. This notion generalizes the multivalued maps considered in \cite{barmak2020Lefschetz}, i.e., every susc multivalued map $F:X\multimap Y$ between finite $T_0$ topological spaces with acyclic values is a Vietoris-like multivalued map, see Proposition \ref{prop:suscWeaktoapointisvietoriss}.  The concept of Vietoris-like map is a finite analogous of the classical definition of Vietoris map. Indeed, it also satisfies that induces isomorphisms in homology groups. Then, we can obtain a coincidence theorem, as in the classical setting.
\begin{thmx}[Coincidence theorem]\label{thm:GeneralizacionCoincidenceThm} Let $f,g: X\rightarrow Y$ be continuous maps between finite $T_0$ topological spaces, where $f$ is a Vietoris-like map, then $\Lambda(g_* \circ  f^{-1}_*)$ is defined and if $\Lambda(g_* \circ  f^{-1}_*)\neq 0$, there exists $x\in X$ such that $f(x)=g(x)$.
\end{thmx}
From here, the classical Lefschetz fixed point theorem for finite topological spaces and single valued maps can be deduced since the identity map is a Vietoris-like map. We generalize the notion of Vietoris-like map to multivalued maps so as to obtain a new Lefschetz fixed point theorem. One advantage of this notion is the fact that it is very flexible. In fact, there is no kind of continuity required in the definition of a Vietoris-like multivalued map. Despite the previous fact, Vietoris-like multivalued maps induce morphisms in homologoy groups. Then, we obtain a generalization of Theorem \ref{thm:lefschetzBarmakMrozek}.
\begin{thmx}[Lefschetz fixed point theorem for multivalued maps]\label{thm:weakLefschetzFixedPointTheorem} Given a finite $T_0$ topological space $X$. If $F:X\multimap X$ is a Vietoris-like multivalued map and $\Lambda(F_*)\neq 0$, then there exists $x\in X$ with $x\in F(x)$.
\end{thmx}
In general, the composition of two Vietoris-like multivalued maps is not a Vietoris-like multivalued map. But, the composition of Vietoris-like multivalued maps presents a good behavior in terms of the Lefschetz fixed point theorem. Specifically,
\begin{thmx}\label{thm:LeftschetzComposiciones} Let $F:X\multimap X$ be a multivalued map, where $X$ is a finite $T_0$ topological space. Suppose that $F =G_n  \circ \cdots  \circ G_0$, where $G_i:Y_i \multimap Y_{i+1}$, $Y_0=Y_{n+1}=X$, $Y_i$ is a finite $T_0$ topological space and $G_i$ is a Vietoris-like multivalued map. If $\Lambda(G_{n*}\circ \cdots \circ  G_{0*})\neq 0$, then there exists a point $x\in X$ such that $x\in F(x)$.
\end{thmx}

The organization of the paper is as follows. In Section \ref{sec:vietorismap}, we introduce the notion of Vietoris-like for single valued maps and multivalued maps and we present examples. In Section \ref{sec:coincidence}, a coincidence theorem for finite topological spaces is obtained, Theorem \ref{thm:GeneralizacionCoincidenceThm}. Then, Lefschetz fixed point theorems are deduced, e.g., Theorem \ref{thm:weakLefschetzFixedPointTheorem}. Moreover, it is introduced the notion of continuous selector for a multivalued map. An existence result of selectors for a certain class of multivalued maps is obtained. Finally, coincidence theorems for multivalued maps are given. In Section \ref{sec:nonacyclic}, hypothesis regarding to the image of the mutlivalued maps considered in previous sections are relaxed. A Lefschetz fixed point theorem is obtained for this new class of multivalued maps, Theorem \ref{thm:LeftschetzComposiciones}. To conclude, in Section \ref{sec:application}, it is proposed a method to approximate a dynamical system by finite topological spaces using the theory developed previously.

\section{Vietoris-like maps and multivalued maps}\label{sec:vietorismap}
In this section, we introduce the notion of Vietoris-like for single valued maps and multivalued maps, we also present some examples and properties. 
\begin{df}\label{def:vietoris-likemap} Let $f:X\rightarrow Y$ be a continuous function between two finite $T_0$ topological spaces, we say that $f$ is a Vietoris-like map if for every chain $y_1<y_2<...<y_n$ in $Y$, we have that $\bigcup_{i=1}^n f^{-1}(y_i)$ is acyclic.
\end{df}
\begin{rem} Definition \ref{def:vietoris-likemap} implies the surjectivity of the maps considered. Therefore, $f:X\rightarrow X$ is a Vietoris-like map if and only if $f$ is a homeomorphism. Moreover, if $f:X\rightarrow Y$ is a Vietoris-like map, then $f$ is also a Vietoris-like map when it is considered the other possible partial order on $X$ and $Y$ at the same time. 
\end{rem}

\begin{thm}\label{thm:vietorisTheorem}
If $f:X\rightarrow Y$ is a Vietoris-like map, then $f$ induce isomorphisms in all homology groups. 
\begin{proof}

The idea of the proof is to use \cite[Corollary 6.5]{barmak2011onQuillen}, which says that if $\varphi:X\rightarrow Y$ is a continuous map between finite $T_0$ topological spaces satisfying that $\mathcal{K}(\varphi^{-1}(U_y))$ (or equivalently $\varphi^{-1}(U_y)$) is acyclic for every $y\in Y$ , then $\mathcal{K}(\varphi)$ ( or $\varphi$) induces isomorphism in all homology groups. On the one hand, we have $\mathcal{K}(f)^{-1}(\overline{\sigma})=\mathcal{K}(\bigcup_{i=1}^n f^{-1}(y_i))$, where $\sigma$ is a simplex given by some chain $y_1<...<y_n$ in $Y$ and $\overline{\sigma}$ denotes the subcomplex of $\mathcal{K}(Y)$ given by $\sigma$ and all its faces, i.e., all the possible subchains of $y_1<...<y_n$. We prove the last assertion, if $\tau\in \mathcal{K}(f)^{-1}(\overline{\sigma})$, then $\mathcal{K}(f)(\tau)\subseteq \overline{\sigma}$. Hence, $\mathcal{K}(f)(\tau)$ is given by a subchain of $y_1<...<y_n$, which implies that $\tau \in \mathcal{K}(\bigcup_{i=1}^n f^{-1}(y_i))$. We prove the other content, if $\tau\in  \mathcal{K}(\bigcup_{i=1}^n f^{-1}(y_i))$, then $\tau$ is given by a chain $x_1<...<x_m$, where $x_i\in f^{-1}(y_j)$ for some $j=1,...,n$ and $i=1,...,m$. $\mathcal{K}(\tau)$ is given by the chain $f(x_1)\leq...\leq  f(x_n)$, which is a subchain of $y_1<...<y_n$. Therefore, $\mathcal{K}(f)(\tau)\in \overline{\sigma}$.

By hypothesis,  $\mathcal{K}(f)^{-1}(\overline{\sigma})$ is acyclic due to the equality that we prove above. On the other hand, we have $\mathcal{X}(\mathcal{K}(f)^{-1}(\overline{\sigma}))=\mathcal{X}(\mathcal{K}(f))^{-1}(U_\sigma)$, where $U_\sigma$ is the minimal open set of $\sigma\in \mathcal{X}(\mathcal{K}(Y))$, that is to say, $U_\sigma$ consists of all subsimplices of $\sigma$. Therefore, we are in the hypothesis of \cite[Corollary 6.5]{barmak2011onQuillen} because for every $\sigma\in \mathcal{X}(\mathcal{K}(Y))$, we get that $\mathcal{X}(\mathcal{K}(f))^{-1}(U_\sigma)$ is acyclic. Thus, $\mathcal{X}(\mathcal{K}(f))$ induces isomorphism in all homology groups. By Theorem \ref{thm:McCord1} and Theorem \ref{thm:McCordX}, it can be deduced that $f$ induces isomorphisms in all homology groups.
\end{proof}

\end{thm}

Theorem \ref{thm:vietorisTheorem} is a finite analogue of the Vietoris-Beggle mapping theorem. It also justifies Definition \ref{def:vietoris-likemap}, in the following example, we show that it is not possible to obtain an analogue of a Vietoris-Begle mapping theorem (Theorem \ref{thm:vietorisClasico}) for finite topological spaces if we keep an analogue of the classical definition of Vietoris map (Definition \ref{def:Vietorismapclasico}).

\begin{ex}\label{ex:novietoristheoremfinitespaces} We consider $X=\{A,B,C,D,E,F \}$ and the following relations $A>C,D,E,F$; $B>C,D,E,F$; $C>,E,F$ and $D>,E,F$, i.e., $X$ is the minimal finite model of the $2$-dimensional sphere \cite{barmak2007minimal}. On the other hand, we consider $Y=\{ M,N\}$ and the relation $M<N$. We define $f$ as follows: $f(A)=f(B)=f(C)=N$ and $f(D)=f(E)=f(F)=M$. Clearly, $f$ is a continuous surjective function, we also have that $f^{-1}(M)$ and $f^{-1}(N)$ are contractible. But, $f$ does not induce isomorphism in homology, $X$ is weak homotopy equivalent to a $2$-dimensional sphere and $Y$ is weak homotopy equivalent to a point.
\begin{figure}[h]
\centering
\includegraphics[scale=0.9]{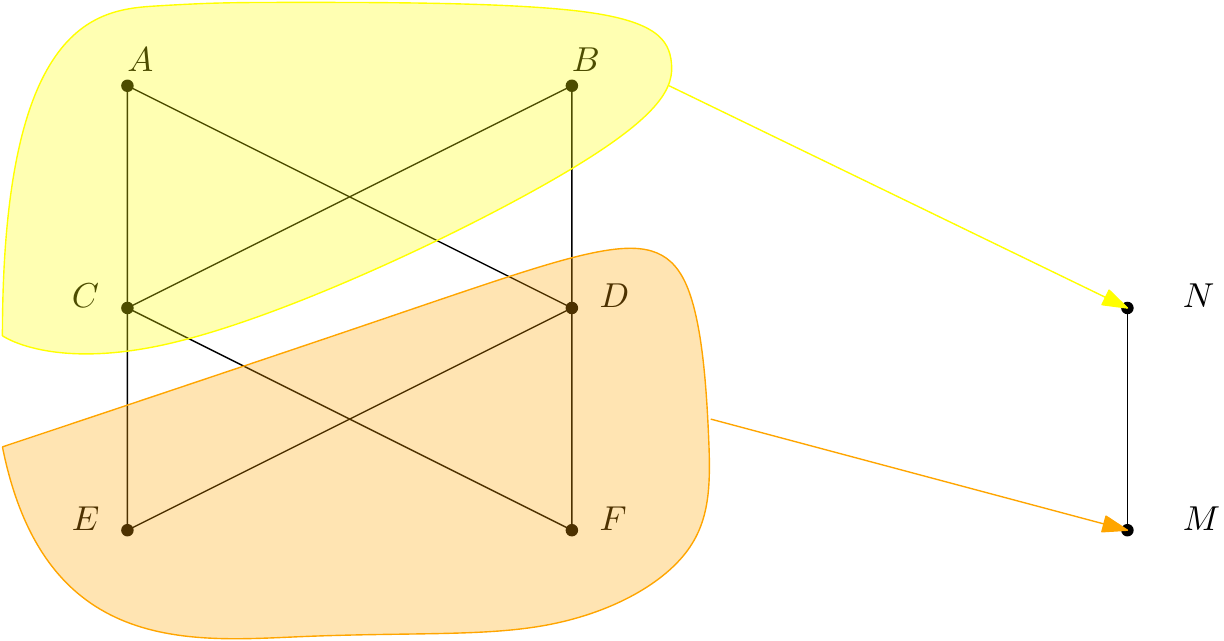}
\caption{Schematic description of $f$ on the Hasse diagrams of $X$ and $Y$.}

\end{figure}
\end{ex}

\begin{lem}\label{lem:vietoriscomposicion} If $f:X\rightarrow Y$ and $g:Y\rightarrow Z$ are two Vietoris-like maps between finite $T_0$ topological spaces, then the composition $g\circ f$ is also a Vietoris-like map.
\begin{proof}
The continuity of $g\circ f$ is trivial. Let us take a chain  $z_1<...<z_n$ in $Z$, $\bigcup_{i=0}^n g^{-1}(z_i)$ is acyclic by hypothesis. If $A=\bigcup_{i=0}^n g^{-1}(z_i)$, then $f_{|f^{-1}(A)}:f^{-1}(A)\rightarrow A$ is a Vietoris-like map trivially. Therefore, by Theorem \ref{thm:vietorisTheorem}, we get that $f^{-1}(A)$ and $A$ have the same homology groups, which implies that $f^{-1}(A)$ is acyclic.
\end{proof}
\end{lem}

\begin{lem}\label{lem:2out3vietorisIntento} Let $f:X\rightarrow Y$ and $g:Y\rightarrow Z$ be continuous maps between finite $T_0$ topological spaces. If $f$ and $g\circ f$ are Vietoris-like maps, then $g$ is also a Vietoris-like map.
\begin{proof}
We denote $h=g\circ f$. We take a chain $z_1<z_2<...<z_n$ in $Z$. By hypothesis, $\bigcup_{i=1}^n h^{-1}(z_i)$ is acyclic, we denote $A=\bigcup_{i=1}^n g^{-1}(z_i)$. We get that $f_{|f^{-1}(A)}:f^{-1}(A)\rightarrow A$ induces isomorphisms in all homology groups because $f$ is a Vietoris-like map. Then, $A$ is acyclic and $g$ is also a Vietoris-like map. 
\end{proof}
\end{lem}

Despite Lemma \ref{lem:vietoriscomposicion} and Lemma \ref{lem:2out3vietorisIntento}, the 2-out-of-3 property does not hold for Vietoris-like maps in general as we show with the following example. 
\begin{ex}\label{ex:2de3noEsCierta} We consider $X=\{A,B,C \}$ with the following partial order: $A,B<C$. We also consider $Y=\{D,E\}$, where we declare $D<E$. Finally, $Z=\{ F\}$. We define $f:X\rightarrow Y$ given by $f(A)=D,f(B)=D$ and $f(C)=E$. We consider $g:Y\rightarrow Z$ as the constant map. It is trivial to check that $g$ and $g\circ f$ are Vietoris-like maps. On the other hand, $f^{-1}(D)=\{A,B \}$ has the weak homotopy type of the disjoint union of two points. Then, $f$ is not a Vietoris-like map. 
\end{ex}

In addition, if $f:X\rightarrow Y$ is a Vietoris-like map and $g:X\rightarrow Y$ is a weak homotopy equivalence which is homotopic to $f$, then $g$ is not necessarily a Vietoris-like map. 
\begin{ex} We consider $X=\{A,B \}$, where we declare $A>B$, and $W= \{C,$ $D,$ $E,$ $F,G,$ $H,I,J,K \}$, where we declare $C>F,G,I,J,K$; $D>F,H,I,J,K$; $E>G,H,I,J,K$; $F>I,J$; $G>I,J,K$ and $H>J,K$. $W$ is the finite $T_0$ topological space introduced in \cite[Figure 2]{rival1976afixed}. $W$ is weak homotopy equivalent to a point but it is not contractible. We consider $f:W\rightarrow X$ given by $f(C)=f(E)=f(G)=A$ and $f(D)=f(F)=f(H)=f(I)=f(J)=f(K)=B$ and $g:W\rightarrow X$ given by $g(C)=g(E)=A$ and $g(W\setminus \{C,E \})=B$. It is easy to check that $f$ and $g$ are continuous maps. $f$ is a Vietoris-like map since $f^{-1}(A)$ is the minimal closed set containing $G$ ($F_G$) and $f^{-1}(B)$ is the minimal open set containing $D$ ($U_D$), which means that $f^{-1}(A)$ and $f^{-1}(B)$ are contractible. Furthermore, $g$ is homotopic to $f$ because $g\leq f$. But $g$ is not a Vietoris-like map, $g^{-1}(A)$ does not have the same weak homotopy type of a point since it is not connected. In Figure \ref{fig:ejemplohomotopicasnovietoris}, we have represented the Hasse diagrams of $X$ and $W$.
\begin{figure}[h]
\centering
\includegraphics[scale=1.2]{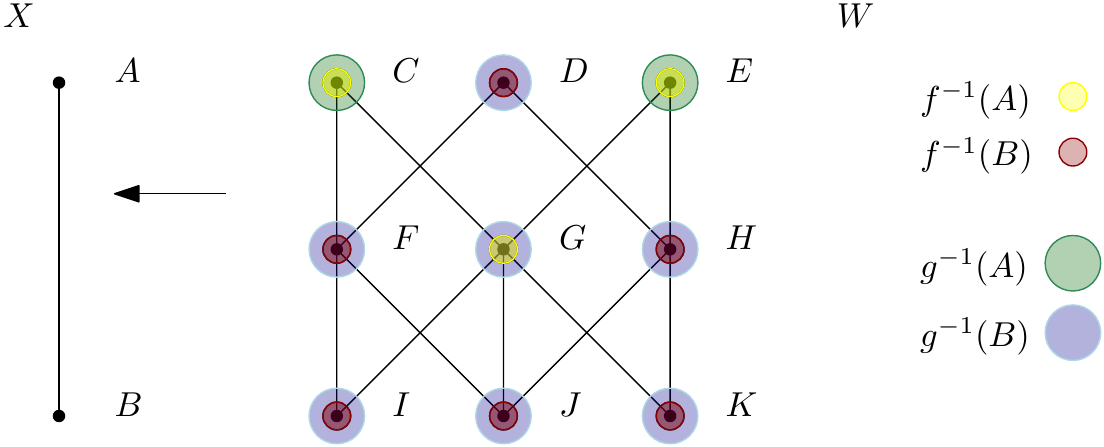}
\caption{Hasse diagrams of $W$ and $X$ and schematic representations of $f^{-1}(A)$, $f^{-1}(B)$, $g^{-1}(A)$ and $g^{-1}(B)$.}\label{fig:ejemplohomotopicasnovietoris}
\end{figure}

\end{ex}


Now, we provide one of the main definitions.
\begin{df}\label{def:vietorislikemulti} Let $F:X\multimap Y$ be a mutlivalued map between finite $T_0$ topological spaces, $F$ is a Vietoris-like multivalued map if the projection $p$ onto the first coordinate from the graph of $\Gamma(F)$ is a Vietoris-like map.
\end{df}
It is important to observe that we do not require any notion of continuity in the multivalued maps of Definition \ref{def:vietorislikemulti}. 
\begin{rem}\label{rem:imagevietorislikemultiisweakpoint} If $F:X\multimap Y$ is a Vietoris-like multivalued map between finite $T_0$ topological spaces, then $F(x)$ is acyclic for every $x\in X$. This is due to the fact that $p^{-1}(x)=x\times F(x)$.
\end{rem}
Furthermore, the compositions of Vietoris-like multivalued maps are not in general Vietoris-like multivalued maps as we show in the following example. Before the example, we recall the definition of the composition of multivalued maps. If $F:X\multimap Y$ and $G:Y\rightarrow Z$ are multivalued maps, then $G\circ F$ is given by $G(F(x))=\bigcup_{y\in F(x)}G(y)$ .
\begin{ex}\label{ex:composicionNoesVietorislikemulti} We consider $X=\{ A,B,C,D\}$ with the following partial order given as follows $A>C,D$ and $B>,C,D$, i.e., $X$ is the minimal finite model of the circle. $F:X\multimap X$ is given by $F(A)=\{A,C,D \}$, $F(B)=\{B,C,D \}$, $F(C)=C$ and $F(D)=D$. $G:X\multimap X$ is given $G(A)=A$, $G(B)=B$, $G(C)=\{C,A,B \}$ and $G(D)=\{D,A,B \}$. It is easy to check that $F$ and $G$ are Vietoris-like multivalued maps. On the other hand, $G\circ F:X\multimap X$ is not a Vietoris-like multivalued map since $G(F(A))=X$ is not acyclic, which implies a contradiction with Remark \ref{rem:imagevietorislikemultiisweakpoint}. In addition, $G$ is an example of a Vietoris-like multivalued map which is not usc.
\end{ex}

On the other hand, we get

\begin{lem}\label{lem:ComposicionNormalYVietorisesVietoris} If $f:X\rightarrow Y$ is a continuous map and $G:Y\multimap  Z$ is a Vietoris-like multivalued map, where $X,Y$ and $Z$ are finite $T_0$ topological spaces, then $G\circ f:X\multimap Z$ is a Vietoris-like multivalued map.
\end{lem}
\begin{proof}
By hypothesis, $G$ is a Vietoris-like multivalued map. Then,  $p_G$ is a Vietoris-like map, where $p_G$ denotes the projection onto the first coordinate of the graph of $G$. We consider a chain $x_1<...<x_n$ in $X$, by the continuity of $f$, $f(x_1)\leq...\leq f(x_n)$ is a chain in $Y$. Hence, $A=\bigcup_{i=1}^np_G^{-1}(f(x_i))$ is acyclic. One  point in $A$ is of the form $(f(x_i),y)$, where $y\in G(f(x_i))$. We will prove that $p_{G\circ f}$ is a Vietoris-like map, where $p_{G\circ f}$ denotes the projection onto the first coordinate of the graph of $G\circ f$. Then, we need to show that $B=\bigcup_{i=1}^n p_{G\circ f}^{-1}(x_i)$ is acyclic. To do that, we verify that $A$ and $B$ are homotopy equivalent. One point in $B$ is of the form $(x_i,y)$, where $y\in G(f(x_i))$. We define $L:B\rightarrow A$ given by $L(x_i,y)=(f(x_i),y)$. $L$ is trivially well defined and continuous. We consider $R:A\rightarrow B$ given by $R(z,y)=( f_{min}^{-1}(z),y)$, where $f_{min}^{-1}(z)$ denotes the minimum of the intersection of $f^{-1}(z)$ with $\{x_1,...,x_n \}$. $R$ is well defined since $f_{min}^{-1}(z)$ is a non-empty subset of a totally ordered set $(x_1<...<x_n)$. If $(z,y),(z',y')\in A$ with $(z,y)\leq (z',y')$, then $R(z,y)=(f_{min}^{-1}(z),y)\leq (f_{min}^{-1}(z'),y')=R(z',y')$. We argue by contradiction, suppose $f_{min}^{-1}(z)> f_{min}^{-1}(z')$, $z=f(f_{min}^{-1}(z))\geq f(f_{min}^{-1}(z'))=z'$, but $z\leq z'$. Therefore, the only possibility is $z=z'$ and we get $f_{min}^{-1}(z)=f_{min}^{-1}(z')$, which entails a contradiction. It is easy to check that $L\circ R$ is the identity map in $A$. Finally, if $(x_i,y)\in B$, $R(L(x_i,y))=R(f(x_i),y)=(min(f^{-1}(f(x_i))\cap\{x_1,...,x_n\}),y)\leq (x_i,y)$, which means that $R\circ L$ is homotopic to the identity map in $B$. Thus, $A$ and $B$ are homotopy equivalent. 
\end{proof}

If $f:Y\rightarrow Z$ is a continuous map between finite $T_0$ topological spaces and $F:X\multimap Y$ is a Vietoris-like multivalued map between finite $T_0$ topological spaces, the composition $f\circ F$ is defined as follows: $$f(F(x))=\bigcup_{y\in F(x)}f(x).$$
In the previous conditions, it is not possible to get an analogue of Lemma \ref{lem:ComposicionNormalYVietorisesVietoris} as we prove in the following example.
\begin{ex} We consider the finite topological space of one point, $X=\{A \}$. We also consider $Y=\{B,C,D,E \}$ with the following relation $B<C>D<E$ and $Z=\{F,G,H,I \}$ such that $F>H,I$ and $G>H,I$. $F:X\multimap Y$ is given by $F(A)=Y$. Then, $F$ is clearly a Vietoris-like multivalued map. $f:Y\rightarrow Z$ is given by $f(B)=H$, $f(D)=I$, $f(C)=F$ and $f(E)=G$. It is immediate to get that $f$ is a continuous map. We have $f(F(A))=Z$, which implies that $f(F(A))$ is weak homotopy equivalent to a circle because $Z$ is the minimal finite model of the circle. If we suppose that $f\circ F:X\multimap Z$ is a Vietoris-like multivalued map, we get a contradiction with Remark \ref{rem:imagevietorislikemultiisweakpoint}.
\begin{figure}[h]
\centering
\includegraphics[scale=0.9]{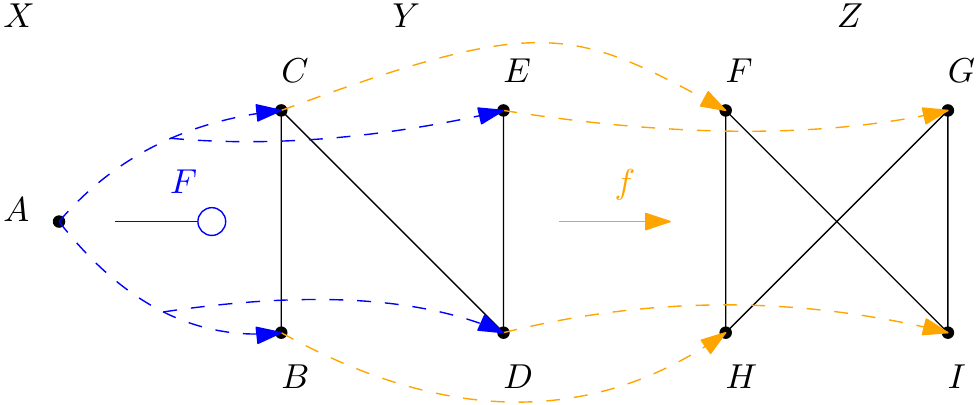}
\caption{Schematic description of $F$ and $f$ on the Hasse diagrams of $X,Y$ and $Z$.}
\end{figure}
\end{ex}

\begin{rem}\label{rem:fcontinuaesmultiVietoris} If $f:X\rightarrow Y$ is a continuous map between finite $T_0$ topological spaces, it is trivial show that the projection of the graph of $f$ onto the first coordinate $p$ is a Vietoris-like map. In fact, $p$ is a homeomorphism. On the other hand, since every continuous map $f$ can be seen as a multivalued map, we get that every continuous map $f$ is a Vietoris-like multivalued map. It is simple to show that $f_*=q_*\circ p_*^{-1}$ because $f\circ p=q$, where $q:\Gamma(f)\rightarrow Y$ denotes the projection onto the second coordinate.
\end{rem}
If $F:X\multimap Y$ is a Vietoris-like multivalued map between  finite $T_0$ topological spaces, $F_*:H_*(X)\rightarrow H_*(Y)$ denotes the morphism induced in homology given by $F_*=q_*\circ p_*^{-1}$, where $q:\Gamma(F)\rightarrow Y$ denotes the projection onto the second coordinate. Since $p$ induces isomorphisms in all homology groups $F_*$ is well defined. If $p$ is not a Vietoris-like map but induces isomorphisms in all homology groups, $F_*$ is also considered as $F_*=q_{*}\circ p_{*}^{-1}$.

\begin{ex}\label{ex:uscnoesVietoris} We consider $X= \{A,B,C,D,E\}$, where $A,B,C<D,E$, and $F:X\multimap X$ given by $F(A)=\{D,B,A \}$, $F(B)=B$, $F(C)=C, F(D)=D$, $F(E)=\{C,E,D\}$. It is immediate that $F$ is a usc multivalued map. Moreover, it is easy to prove that $F$ is not a Vietoris-like multivalued map since $p^{-1}(A)\cup p^{-1}(E)$ is weak homotopy equivalent to a wedge sum of two circles. Then, the property of being usc does not imply the property of being a Vietoris-like multivalued map.
\begin{figure}[h]
\centering
\includegraphics[scale=0.7]{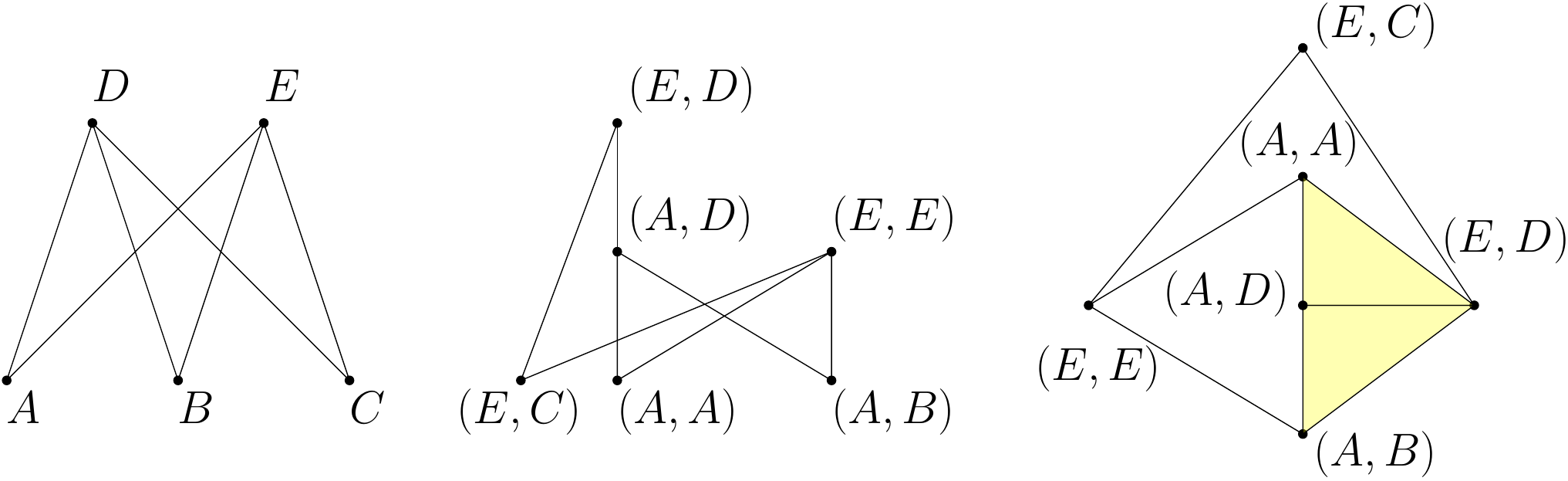}
\caption{From left to right, Hasse diagrams of $X$ and $p^{-1}(A)\cup p^{-1}(E)$ and McCord complex of $p^{-1}(A)\cup p^{-1}(E)$.}\label{fig:ejemplouscnovietoris}
\end{figure}
\end{ex}

As examples of Vietoris-like multivalued maps, we have the multivalued maps considered in \cite{barmak2020Lefschetz}, see the hypothesis of Theorem \ref{thm:lefschetzBarmakMrozek}. In terms of continuity, Vietoris-like multivalued maps are more flexible than the ones mentioned above since the properties of being susc or slsc are not considered to prove Theorem \ref{thm:weakLefschetzFixedPointTheorem} (Lefschetz fixed point theorem for multivalued maps) in Section \ref{sec:coincidence}. It is easy to find Vietoris-like multivalued maps that are not susc with acyclic values. For instance, every continuous map $f:X\rightarrow Y$ can be seen as a Vietoris-like multivalued map, Remark \ref{rem:fcontinuaesmultiVietoris}. For more examples, see the multivauled map $G$ considered in Example \ref{ex:composicionNoesVietorislikemulti} or the following propositions. The multivalued map considered in Proposition \ref{prop:vietorislikempainducenmultivevaluadasvietoris} cannot be clearly a susc multivalued map unless it is constant.
\begin{prop}\label{prop:suscWeaktoapointisvietoriss} If $F:X\multimap Y$ is a multivalued map susc (slsc) such that $F(x)$ is acyclic for every $x\in X$, then $F$ is a Vietoris-like multivalued map.
\begin{proof}
The continuity of $p$ is trivial to show. We follow similar techniques than the ones used in \cite[Lemma 4.1]{barmak2020Lefschetz} so as to get that $p$ is a Vietoris-like map. Let us take a chain $x_1<...<x_n$ in $X$, we denote $A=\bigcup_{i=1}^n p^{-1}(x_i)$. The idea is to show that $A$ has the same homotopy type of $F(x_n)$. We define $i:F(x_n)\rightarrow A$ given by $i(z)=(x_n,z)$, clearly, $i$ is well defined and continuous. We also consider $r:A\rightarrow F(x_n)$ given by $r(x_i,y)=y$, we get by hypothesis that $x_i\leq x_n$ implies $F{(x_i)}\subseteq F(x_n)$, so $r$ is well defined, the continuity of $r$ is trivial to check. It is immediate that $r\circ i=id_{F(x_n)}$. On the other hand, $i\circ r\simeq id_{A}$ because for every $(x_i,y)\in A$ we get that $i(r(x_i,y))=i(y)=(x_n,y)\geq (x_i,y)=id_{A}(x_i,y)$. Thus, $A$ is homotopy equivalent to $F(x_n)$, which implies that $F$ is a Vietoris-like multivalued map.

The second case, the one in parenthesis, is analogous. If we change the partial order on $X$ and $Y$ at the same time, we are in the previous hypothesis and the result follows immediately.
\end{proof}
\end{prop}

\begin{prop}\label{prop:vietorislikempainducenmultivevaluadasvietoris} If $f:X\rightarrow Y$ is a Vietoris-like map, then $F:Y\multimap X$ given by $F(y)=f^{-1}(y)$ is a Vietoris-like multivalued map.
\begin{proof}
We need to show that the projection of the graph of $F$ onto the first coordinate is a Vietoris-like map. We take a chain $y_1<...<y_n$ in $Y$. We denote $A=\bigcup_{i=1}^n p^{-1}(y_i)$. We define $g:A\rightarrow \bigcup_{i=1}^n f^{-1}(y_i)$ given by $g(y_i,z_i)=z_i$, where we have that $z_i\in f^{-1}(y_i)$ for some $y_i$, so $g$ is well defined. The continuity of $g$ follows trivially. Now, we prove that $g$ is injective. We take $(y_i,z), (y_j,w)\in A$ satisfying that $(y_i,z)\neq (y_j,w)$. We have two options. The first one is $y_i\neq y_j$, then, it is clear that $z\neq w$ because $f(z)=y_i$ and $f(w)=y_j$. We deduce that $g(y_i,z)\neq g(y_j,w)$. The second option, $y_i=y_j$ and $z\neq w$, implies $z=g(y_i,z)\neq g(y_j,w)=w$. 
We define $t:\bigcup_{i=1}^n f^{-1}(y_i) \rightarrow A$ as follows: $t(z)=(y_i,z)$, where $z\in f^{-1}(y_i)$ for some $y_i$ in the chain $y_1<...<y_n$. We have that $t$ is well defined since $f$ is a map. Suppose $z\leq w$, where $z,w\in\bigcup_{i=1}^n f^{-1}(y_i)$, $t(z)=(y_i,z)$ and $t(w)=(y_j,w)$ for some $y_i$ and $y_j$, the chain $y_1<...<y_n$ is a totally ordered set, so $y_i\leq y_j$ or $y_i>y_j$. Suppose that $y_i>y_j$ holds. By the continuity of $f$ we obtain a contradiction because $y_i=f(z)\leq f(w)=y_j$. Therefore, the only possibility is $y_i\leq y_j$, which implies that $t(z)=(y_i,z)\leq (y_j,w)=t(w)$. Hence, $t$ is a continuous map. It is easy to check that $t\circ g$ and $g\circ t$ are the identity map in $A$ and $\bigcup_{i=1}^n f^{-1}(y_i)$, respectively. Thus, $g$ is indeed a homeomorphism. From here, we get that $A$ is acyclic and $p$ is a Vietoris-like map.
\end{proof}
\end{prop}

\begin{prop}\label{prop:uscmaximumisvietoris} If $F:X\multimap Y$ is a usc (resp. slc) multivalued map with $F(x)$ containing a maximum (resp. minimum) for every $x\in X$, then $F$ is a Vietoris-like multivalued map.
\begin{proof}
We prove the result for the first case. We take a chain $x_1<...<x_n$ in $X$ and we denote $A=\bigcup_{i=1}^n p^{-1}(x_i)$. We argue as we did in the proof of Proposition \ref{prop:suscWeaktoapointisvietoriss}. We denote by $\overline{x}_i$ the maximum of $F(x_i)$ for $i=1,...,n$. We define $f:A\rightarrow A$ as follows $f(x_i,z)=(x_n,\overline{x}_n)$. $f$ is trivially a continuous map. We check that $f\geq id_A $. By hypothesis, $F$ is usc, therefore, for every $z\in F(x_i)$ there exists $w\in F(x_n)$ such that $z\leq w$. On the other hand, $F(x_n)$ contains a maximum, so $z\leq w\leq \overline{x}_n$. From here, we deduce the desired result. 

For the second case, we only need to change the partial order on $X$ and $Y$ at the same time. Then, we are in the hypothesis of the first case.
\end{proof}
\end{prop}

With the following example, we show that if $F:X\multimap Y$ is a usc multivalued map such that $F(x)$ contains a minimum for every $x\in X$, then $F$ is not necessarily a Vietoris-like mutlivalued map.

\begin{ex}\label{ex:uscminimumnotvietoris} We consider $X=\{H,I \}$ with $H<I$ and $Y=\{A,B,C,D \}$ with $A>C,D$ and $B>C,D$, that is to say, a finite model of the unit interval and the circle. We define $F:X\multimap Y$ given by $F(H)=\{ D\}$ and $F(I)=\{A,B,C \}$. $F$ is clearly usc and satisfies that for every $x\in X$ there exists a minimum in $F(x)$. The finite topological space associated to the graph of $F$ is given by $\Gamma(F)=\{(I,A), (I,B), (I,C), (H,D) \}$ with the following partial order: $(I,A)>(I,C), (H,D)$ and $(I,B)>(I,C),(H,D)$. Hence, $\Gamma(F)$ is a finite model of the circle. Thus, the projection onto the first coordinate of the graph of $F$ is not a Vietoris-like map.
\end{ex}

\begin{prop}\label{prop:secondprojectionVietoris} If $f:X\rightarrow Y$ is a continuous map between finite $T_0$ topological spaces, the second projection of the graph of $F:f(X)\multimap X$ given by $F(y)=f^{-1}(y)$ is a Vietoris-like map.
\begin{proof}
We denote by $q:\Gamma(F)\rightarrow X$ the projection onto the second coordinate. We have trivially that $q$ is surjective, for every $x\in X$ we get $(f(x),x)\in \Gamma(F)$. If $z,t\in q^{-1}(x)$ for some $x\in X$, we get $f(x)=z$ and $f(x)=t$, so $z=t$. Then, the cardinality of $q^{-1}(x)$ is one. Let us take a chain $x_1<x_2<...<x_n$ in $X$. Hence, we need to show that $A=\bigcup_{i=1}^n q^{-1}(x_i)$ is acyclic. We check that $(t_n,x_n)=q^{-1}(x_n)\in A$ is a maximum. If $(t_i,x_i)=q^{-1}(x_i)$ with $i<n$, we have $f(x_i)=t_i$ and $f(x_n)=t_n$. By the continuity of $f$, we deduce that $t_i\leq t_n$, which implies $q^{-1}(x_n)\geq q^{-1}(x_i)$. Thus, $A$ is contractible.
\end{proof}
\end{prop}

\section{A coincidence theorem for finite topological spaces}\label{sec:coincidence}

In this section, we prove a coincidence theorem for finite topological spaces. As a consequence, we will obtain some versions of the Lefschetz fixed point theorem.

\begin{prop}\label{prop:McCordCoincidenFinitoCoinciden} Given two  finite $T_0$ topological spaces $X,Y$ and two continuous functions $f,g:X\rightarrow Y$. If there exists $x\in |\mathcal{K}(X)|$ such that $|\mathcal{K}(f)|(x)=|\mathcal{K}(g)|(x)$, then there exists $y\in X$ with $f(y)=g(y)$.
\begin{proof}
We denote $f'=|\mathcal{K}(f)|$ and $g'=|\mathcal{K}(g)|$ for simplicity. By \cite[Theorem $2$]{mccord1966singular}, we have the following relations, $f_Y\circ f' =f \circ f_X$ and $f_Y\circ g' =g \circ f_X$ , where $f_X:|\mathcal{K}(X)|\rightarrow X$ and $f_Y:|\mathcal{K}(Y)|\rightarrow Y$ are weak homotopy equivalences. Let us take $x\in |\mathcal{K}(X)|$ such that $f'(x)=g'(x)$. Using the previous relations we get:
$$f(f_X(x))=f_Y(f'(x))=f_Y(g'(x))=g(f_X(x)). $$
Therefore, $y=f_X(x)\in X$ satisfies that $f(y)=g(y)$, as we wanted.
\end{proof}
\end{prop}
Now, we prove one of the main results. 

\begin{proof}[\textbf{Proof of Theorem \ref{thm:GeneralizacionCoincidenceThm}}]

By Theorem \ref{thm:vietorisTheorem}, $\Lambda(g_* \circ  f^{-1}_*)$ is well defined. For the second part, we argue by contradiction. Suppose that $f$ and $g$ do not have coincidence points, i.e., for every $x\in X$, we have $f(x)\neq g(x)$.  
We define an acyclic carrier $\Phi:\mathcal{K}(Y)\rightarrow \mathcal{K}(X)$, that is to say, $\Phi$ is a function which assigns an acyclic subcomplex of $\mathcal{K}(X)$  for each simplex $\sigma\in \mathcal{K}(Y)$ and satisfies that if $\sigma\subseteq \tau$, then $\Phi(\sigma)\subseteq \Phi(\tau)$. For every simplex $\sigma\in \mathcal{K}(Y)$, where $\sigma$ is given by a chain $y_1<...<y_n$ of $Y$, we define $\Phi(\sigma)=\mathcal{K}(\bigcup_{i=1,...,n}f^{-1}(y_i))$. By hypothesis, $f$ is a Vietoris-like map, so $\Phi(\sigma)$ is acyclic. Furthermore, by construction, it is clear that if $\tau\subseteq \sigma$, we have $\Phi(\tau)\subseteq \Phi(\sigma)$. By the acyclic carrier theorem (see for instance \cite[Theorem 13.3]{munkres1984elements}), there exists a chain-map $\phi:C_*(\mathcal{K}(Y))\rightarrow C_*(\mathcal{K}(X))$, which is carried by $\Phi$, i.e., for every $n$-simplex $\sigma$ in $\mathcal{K}(Y)$, $\phi(\sigma)$ is a linear combination of simplices of $C_n(\Phi(\sigma))$. We will show inductively that $\mathcal{K}(f)_{\#}(\phi(\sigma))=\sigma$ for every $\sigma\in \mathcal{K}(Y)$, where $\mathcal{K}(f)_{\#}$ denotes the chain-map induced by the simplicial map $\mathcal{K}(f):\mathcal{K}(X)\rightarrow \mathcal{K}(Y)$. Without loss of generality, we can assume that $\mathcal{K}(f)_{\#}(\phi(v))=v$ for every vertex in $\mathcal{K}(Y)$, since $\phi(v)$ can be taken as a vertex in $\mathcal{K} (f)^{-1}(v)$, see \cite[Proof of Theorem 13.3]{munkres1984elements}. If $e$ is an edge, we have $\partial (\mathcal{K}(f)_{\#}(\phi(e)))=\partial e$. We also know that $\phi(e)$ is a linear combination of edges of $\Phi(e)$, that is to say, $\phi(e)=\sum k_i\tau_i$, where $k_i$ is a coefficient and $\tau_i$ is an edge in $\Phi(e) $ for every $i$. We have that $\mathcal{K}(f)$ is a simplicial map, which implies that $\mathcal{K}(f)_\#$ sends $\tau_i$ to zero or $e$. Therefore, $\mathcal{K}(f)_{\#}(\phi(e))=\sum_j k_j e$. Since $\partial (\mathcal{K}(f)_{\#}(\phi(e)))=\partial e$ we get that $\sum_j k_j=1$  and we deduce $\mathcal{K}(f)_{\#}(\phi(e))=e$. We can follow inductively to prove the result. From here, it is easy to deduce that $\mathcal{K}(f)_\#\circ \phi$ is chain homotopic to the identity chain-map $id_\#:C_*(\mathcal{K}(Y))\rightarrow C_*(\mathcal{K}(Y))$. Then, $\phi$ corresponds in homology to $\mathcal{K}(f)_*^{-1}$.

We consider $\mu=\mathcal{K}(g)_{\#}\circ \phi :C_*(\mathcal{K}(Y))\rightarrow C_*(\mathcal{K}(Y))$. We denote by $\mu_*$ the induced homomorphism in homology. We also have
$$\sum_{i=0}(-1)^i \text{trace}(\mu_i)=\sum_{i=0}(-1)^i \text{trace}(\mu_{*}:H_i(\mathcal{K}(Y))\rightarrow H_i(\mathcal{K}(Y))). $$
See for instance \cite[Proof of Theorem 2C.3]{hatcher2000algebraic}.

If the trace of $\mu_i$ is not zero for some $i$, we get that there exists a simplex $\sigma\in C_i(\mathcal{K}(Y))$ with $\mu_i(\sigma)=k\sigma+...$, where $k$ is a non-zero coefficient. Therefore, there exists $\gamma$ in $\phi(\sigma)$ with $\mathcal{K}(g)_{\#}(\gamma)=\sigma$. We prove the last assertion, we have $\phi(\sigma)=\sum k_j\tau_j$, where $\tau_j\in C_i(\Phi(\sigma))$ is a simplex and $k_j$ is a coefficient for every $j$. Since $\mathcal{K}(g)$ is a simplicial map, $\mathcal{K}(g)_{\#}$ sends $\tau_j$ to an $i$-simplex or zero, which implies the desired assertion. In fact, $|\mathcal{K}(g)|$ restricted to $\overline{\gamma}$ is a homeomorphism, where $\overline{\gamma}\subset |\mathcal{K}(X)|$ denotes the closed simplex given by the simplex $\gamma\in \mathcal{K}(X)$. On the other hand, we get that $|\mathcal{K}(f)|(\overline{\gamma})\subseteq \overline{\sigma}$. Then, by the Brouwer fixed point theorem,  $|\mathcal{K}(f)|\circ |\mathcal{K}(g)|^{-1}_{|_{\overline{\gamma}}}:\overline{\sigma}\rightarrow \overline{\sigma}$ has a fixed point $t$. Therefore, $|\mathcal{K}(g)|^{-1}_{|_{\overline{\gamma}}}(t)$ is a coincidence point for $|\mathcal{K}(f)|$ and $|\mathcal{K}(g)|$. By Proposition \ref{prop:McCordCoincidenFinitoCoinciden}, we get that $f$ and $g$ have a coincidence point, which entails a contradiction.

Previously, we proved that $\mathcal{K}(f)^{-1}_*=\phi_* $. Then, $\mu_*=\mathcal{K}(g)_*\circ \mathcal{K}(f)_*^{-1}$, which implies that $\sum_{i}(-1)^i \text{trace}(\mathcal{K}(g)_*\circ \mathcal{K}(f)_*^{-1})=0$. In addition, $\Lambda(\mathcal{K}(g)_*\circ \mathcal{K}(f)_*^{-1})=\Lambda(\text{trace}(g_*\circ f_*^{-1})$. We prove the last equality. By Theorem \ref{thm:McCord1}, we have $f_*\circ f_{X*}=f_{Y*}\circ \mathcal{K}(f)_*$ and $g_*\circ f_{X*}=f_{Y*}\circ \mathcal{K}(g)_*$. Since $f_*$, $\mathcal{K}(f)_*$, $f_{X*}$ and $f_{Y*}$ are isomorphisms, we get $f_*^{-1}=f_{X*}\circ \mathcal{K}(f)_*^{-1}\circ f_{Y*}^{-1}$ and $g_*=f_{Y*}\circ \mathcal{K}(g)_*\circ f_{X*}^{-1}$. Then, $g_*\circ f^{-1}_*=f_{Y*}\circ \mathcal{K}(g)_*\circ \mathcal{K}(f)_*^{-1}\circ f_{Y*}^{-1}$. By the properties of the trace, we get the desired equality, which entails a contradiction because $\Lambda(g_*\circ f_*^{-1})\neq 0$. Thus, there must exists a point $x\in X$ such that $f(x)=g(x)$.

\end{proof}


Suppose $f,g:X\rightarrow Y$ are continuous maps, where $X$ and $Y$ are finite $T_0$ topological spaces and $f$ is a Vietoris-like map. If $\Lambda(g_*\circ f_*^{-1})\neq 0$ and $g':X\rightarrow Y$ is a continuous map homotopic to $g$, then $g'$ and $f$ have at least one coincidence point. The opposite result does not hold, that is to say, if $f':X\rightarrow Y$ is a continuous map homotopic to $f$, then $f'$ and $g$ do not have necessarily a coincidence point.
\begin{ex} We consider $X=\{A,B,C \}$, where $C>A,B$, $f:X\rightarrow X$ given by $f(C)=C$, $f(A)=B$ and $f(B)=A$, $g:X\rightarrow X$ given by $g(y)=A$ for every $y\in X$ and $f':X\rightarrow X$ given by $f'(y)=B$ for every $y\in X$. It is easy to deduce that $f$ and $f'$ are homotopic. Moreover, $f$ is a Vietoris-like map and $\Lambda(g_*\circ f^{-1}_*)\neq 0$. But, $f'$ and $g$ do not have a coincidence point.
\end{ex}


We can also obtain coincidence theorems for multivalued maps as corollaries of Theorem \ref{thm:GeneralizacionCoincidenceThm}.

\begin{cor}\label{cor:CoincidenciaMultivaluadaYnormal}
Let $f: X\rightarrow Y$ be a continuous map between finite $T_0$ topological spaces such that $f$ is a Vietoris-like map and let  $F:X \multimap Y$ be a  Vietoris-like multivalued map. Then, $\Lambda(F_* \circ  f^{-1}_*)$ is defined and if $\Lambda(F_* \circ  f^{-1}_*)\neq 0$, there exists $x\in X$ such that $f(x)\in F(x)$.
\begin{proof}
We have the following diagram,
\[\begin{tikzcd}
X \arrow{r}{f}  & Y  \\
\Gamma(F) \arrow{u}{p} \arrow{ur}{q} &
\end{tikzcd}
\]
By hypothesis, $p$ is a Vietoris-like map. Then, $f \circ p$ is a Vietoris-like map due to Lemma \ref{lem:vietoriscomposicion}, so we are in the hypothesis of Theorem \ref{thm:GeneralizacionCoincidenceThm}. Therefore, there exists $(x,y)\in \Gamma(F)$ such that $f( p(x,y))=q(x,y)$, so $f(x)=y\in F(x)$.
\end{proof}
\end{cor}

\begin{cor}\label{cor:CoincidenciaMultivaluadaInduceIsomorfismoYnormal}
Let $f: X\rightarrow Y$ be a continuous map between finite $T_0$ topological spaces and let $F:X \multimap Y$ be a multivalued map such that the second projection, $q$, is a Vietoris-like map. Then, $\Lambda( f_* \circ  F^{-1}_*)$ is defined, where $F_*^{-1}=p_*\circ q_*^{-1}$, and if $\Lambda( f_* \circ  F^{-1}_*)\neq 0$, there exists $x\in X$ such that $f(x)\in F(x)$.
\begin{proof}
We have the following diagram, 
\[\begin{tikzcd}
X \arrow{r}{f}  & Y  \\
\Gamma(F) \arrow{u}{p} \arrow{ur}{q} &
\end{tikzcd}
\]
$F_*^{-1}$ is well defined since $q$ induces isomorphisms in all homology groups, Theorem \ref{thm:vietorisTheorem}. We are in the hypothesis of Theorem \ref{thm:GeneralizacionCoincidenceThm}. Thus, there exists a coincidence point $(x,y)\in \Gamma(F)$ with $y=q(x,y)=f(p(x,y))=f(x)$ so $f(x)\in F(x)$.
\end{proof}
\end{cor}


From the theory developed previously, it is easy to get an analogue of the Lefschetz fixed point theorem for Vietoris-like multivalued maps.
\begin{proof}[\textbf{Proof of Theorem \ref{thm:weakLefschetzFixedPointTheorem}}]
$\Lambda(F_*)$ is well defined since $F$ is a Vietoris-like multivalued map. Then, we are in the hypothesis of Corollary \ref{cor:CoincidenciaMultivaluadaYnormal}, where we are considering the identity map as the single valued map. 

\end{proof}
We can also obtain the classical Lefschetz fixed point theorem for finite spaces \cite{baclawski1979fixed} using the previous techniques.
\begin{cor}[Lefschetz fixed point theorem]\label{cor:classicalLefschetzFixedPointTheoremFiniteSpaces} Let $f:X\rightarrow X$ be a continuous function and let $X$ be a finite $T_0$ topological space. If $\Lambda(f)\neq 0$, there exists a fixed point. Furthermore, $\Lambda(f)$ is the Euler characteristic of $Fix(f)$.
\begin{proof}
By Remark \ref{rem:fcontinuaesmultiVietoris}, $f$ can be seen as a Vietoris-like multivalued map. Then, we are in the hypothesis of Theorem \ref{thm:weakLefschetzFixedPointTheorem}.
%
%

The second part follows easily from the fact that $p$ is a homeomorphism. Then, we have $q\circ p^{-1}:X\rightarrow X$ and the following relations:
\begin{align*}
\Lambda(f_*)=\Lambda(q_*\circ p_*^{-1}\}=&\sum_{i=0}(-1)^i \text{trace}(q_*\circ p_*^{-1})= \sum_{i=0}(-1)^i |\{ \sigma\in \mathcal{K}(X)_i| q(p^{-1}(\sigma))=\sigma \}|= \\
 =&\sum_{i=0}(-1)^i |\{ v_0<...<v_i| q(p^{-1}(v_0<...<v_i))=v_0<...<v_i \}|= \\ 
=&\sum_{i=0}(-1)^i |\{ v_0<...<v_i| q(p^{-1}(v_j))=v_j \ \text{where} \ \ j\in \{0,...,i\} \}|=\\
=&\chi(Fix(f)).
\end{align*}
The previous relations are the same used in \cite[Theorem 1.1]{baclawski1979fixed}.
\end{proof}
\end{cor}

A continuous selector for a multivalued map $F:X\multimap Y$ is a continuous function $f:X\rightarrow Y$ such that $f(x)\in F(x)$ for every $x\in X$, where $X$ and $Y$ are finite $T_0$ topological spaces. In \cite{barmak2020Lefschetz}, it is proved that if $F:X\multimap X$ is a susc (slsc) multivalued map with acyclic values and $f$ is a continuous selector, then $\Lambda(f)=\Lambda(F_*)$. With the following proposition, we extend the class of multivalued maps satisfying that $\Lambda(f)=\Lambda(F)$ for every continuous selector $f$ of $F$.

\begin{prop}\label{prop:continuousselectoruscmasmaximum} Let $X$ and $Y$ be finite $T_0$ topological spaces. If $F:X\multimap Y$ is a multivalued map usc (lsc) such that $F(x)$ contains a maximum (minimum) $\overline{x}$ for every $x\in X$, then there exists a continuous selector $f$ for $F$ with $F_*=f_*$. Furthermore, if $X=Y$, $\Lambda(F_*)=\Lambda(f)$. If $g$  is another continuous selector for $F$, $\Lambda(g)=\Lambda(F_*)$.
\begin{proof}
By Proposition \ref{prop:uscmaximumisvietoris}, $F$ is a Vietoris-like multivalued map. Firstly, we construct a continuous selector $f$ for $F$. We consider $f:X\rightarrow Y$ given by $f(x)=\overline{x}\in F(x)$. We prove the continuity of $f$. If $x\leq y$, we know by the usc property that for every $x'\in F(x)$ there exists $y'\in F(y)$ with $x'\leq y'$. Then, there exists $y'\in F(y)$ with $\overline{x}\leq y' \leq \overline{y}$, so $\overline{x}=f(x)\leq f(y)=\overline{y}$. We have the following diagram, where $p$ and $q$ denote the projection from the graph of $F$ onto the first and second coordinates respectively.

\[\begin{tikzcd}
& \Gamma(F) \arrow{dl}{p} \arrow{dr}{q} & \\
X \arrow{rr}{f} & &  Y
\end{tikzcd}
\]

If $(x,z)\in \Gamma(F)$, we have $z=q(x,z)\leq f(p(x,z))=f(x)=\overline{x}\in F(x)$. Therefore, $f\circ p$ is homotopic to $q$. Concretely, we have $f_*= q_*\circ p^{-1}_*=F_*$ since $p$ is a Vietoris-like map. Thus, if $X=Y$, $\Lambda(F_*)=\Lambda(f_*)$. 

Suppose there is another continuous selector $g$ for $F$. We prove that $g\leq f$. Suppose $x\in X$, $g(x)\in F(x)$ and $f(x)\geq y$ for every $y\in F(x)$, which implies that $f(x)\geq g(x)$. Thus, $\Lambda(g)=\Lambda(f)=\Lambda(F_*)$. 

The result in parenthesis follows directly from the previous one. Suppose $F$ is a multivalued map lsc such that $F(x)$ contains a minimum for every $x\in X$. Then, taking the other possible partial order on $X$ and $Y$, we are in the previous conditions.
\end{proof}
\end{prop}

\begin{ex}\label{ex:noSelector} We consider $X=\{A,B,C,D \}$ with the partial order given as follows $A>C,D$ and $B>C,D$, that is to say, $X$ is the minimal finite model of the circle. We consider $Y=\{E,F,G,H,I,J,K,L \}$ with the following relations $E>L<H>D<G>B<F>C<E$. $Y$ is also a finite model of the circle. $T:X\multimap Y$ is defined as follows $T(C)=I$, $T(D)=K$, $T(A)=\{ E,H,L\}$ and $T(B)=\{F,G,J \}$. It is trivial to check that $T$ is usc and the image of every point has a minimum. Furthermore, $T$ does not have any continuous selector, arguing by contradiction the result follows easily.
\end{ex}

%

Let $F,G:X\multimap Y$ be multivalued maps, we say that $F$ and $G$ have a coincidence point if there exists $x\in X$ such that $G(x)\cap F(x)$ is non-empty. Combining Theorem \ref{thm:coincidenciaClasico}, Corollary \ref{cor:CoincidenciaMultivaluadaYnormal}, Corollary \ref{cor:CoincidenciaMultivaluadaInduceIsomorfismoYnormal} and the notion of continuous selector, it is easy to deduce the following result.
\begin{thm} Let $F:X\multimap Y$ and $G:X\rightarrow Y$ be multivalued maps between finite $T_0$ topological spaces.
\begin{enumerate}
\item If $F$ is a Vietoris-like multivalued map and $G$ admits a continuous selector $g$ that is also a Vietoris-like map, then $\Lambda(F_*\circ g^{-1}_* )$ is defined and if $\Lambda(F_*\circ g^{-1}_* )\neq 0$, $F$ and $G$ have a coincidence point.
\item If the projection onto the second coordinate of the graph of $F$ is a Vietoris-like map and $G$ admits a continuous selector $g$, then $\Lambda(g_*\circ F^{-1}_*)$ is defined and if $\Lambda(g_*\circ F^{-1}_*)\neq 0$, $F$ and $G$ have a coincidence point. 
\item If $G$ admits a continuous selector $g $ that is a Vietoris-like map, $F$ admits a continuous selector $f$ and $\Lambda(f_*\circ g ^{-1}_*)\neq 0$, then $F$ and $G$ have a coincidence point.
\end{enumerate}

\end{thm}


\section{A Lefschetz fixed point theorem for the composition of Vietoris-like multivalued maps}\label{sec:nonacyclic}

In Section \ref{sec:coincidence}, given a mutlivalued map $F:X\multimap X$, hypothesis regarding to the image of every point $x\in X$ are required so as to obtain a version of the classical Lefschetz fixed point theorem. We will show that a Lefschetz fixed point theorem can be obtained for some special multivalued maps such that the image of every point is not necessarily acyclic.

\begin{lem}\label{lem:CompositionIsWellDefined}
Let $f:X\rightarrow Y $ be a continuous map and $G:Y\multimap Z$ be a multivalued map such that the projection of its graph onto the first coordinate induce isomorphism in homology, where $X,Y$ and $Z$ are finite $T_0$ topological spaces. If $H=G\circ f$ satisfies that the projection of its graph onto the first coordinate induce isomorphism in homology, then $H_*=G_*\circ f_*$.
\begin{proof}
We only need to follow the same structure for the analogue result obtained in \cite[Theorem 3.8]{powers1970multi}. We denote by $\Gamma(f),\Gamma(G)$ and $\Gamma(H)$ the finite $T_0$ topological spaces given by the graphs of $f,G$ and $H$, moreover, $p_f,q_f,p_G,q_G,p_H$ and $q_H$ denote their respective projections, where $p$ denotes the projection onto the first coordinate and $q$ the projection onto the second coordinate. We also define an auxiliary map from $\Gamma(H)$ to $Y$, $\overline{F}:\Gamma(H)\rightarrow Y$ given by $\overline{F}(x,z)=f(x)$. The continuity of $\overline{F}$ is trivial since $\overline{F}$ is the composition of continuous maps, $p_H$ and $f$. 
Again, $\Gamma(\overline{F})$ denotes the finite $T_0$ topological space given by the graph of $\overline{F}$, we denote by $\phi_1$ and $\phi_2$ their respective projections onto the first and second coordinates. Finally, we define two extra auxiliary maps from  $\Gamma(\overline{F})$, $\phi_3:\Gamma(\overline{F})\rightarrow \Gamma(f)$ given by $\phi_3((x,z),y)=(x,y)$ and $\phi_4:\Gamma(\overline{F})\rightarrow \Gamma(G)$ given by  $\phi_4((x,z),y)=(y,z)$. It is clear that $\phi_3$ and $\phi_4$ are continuous maps because they preserve the order. It is easy to check that the following diagram of continuous maps is commutative. 
\[\begin{tikzcd}
X  & \Gamma(f) \arrow{l}{p_f}\arrow{r}{q_f} & Y & \Gamma(G) \arrow{l}{p_G}\arrow{r}{q_G} & Z  \\
 & & \Gamma(\overline{F}) \arrow{u}{\phi_2}\arrow{ul}{\phi_3}\arrow{ur}{\phi_4}\arrow{u}{\phi_1}\arrow{d}{\phi_1} & &  \\
 & & \Gamma( H )\arrow{uull}{p_H}\arrow{uurr}{q_H} & &
\end{tikzcd}
\]
By the commutativity of the diagram, we can obtain the following equalities:
\begin{equation}\label{equation1}
p_H\circ \phi_1=p_f \circ \phi_3
\end{equation}
\begin{equation}\label{equation2}
\phi_2= q_f \circ \phi_3
\end{equation}
\begin{equation}\label{equation3}
q_H\circ \phi_1=q_G\circ \phi_4
\end{equation}
\begin{equation}\label{equation4}
\phi_2=p_G\circ \phi_4.
\end{equation}
On the other hand, $p_f,p_G,p_H,\phi_1$ induce isomorphism in homology, so we can take its inverses after applying the homological functor to the previous diagram. Therefore, we can deduce from equations (\ref{equation1}) and (\ref{equation2}).
\begin{equation}\label{equation5}
q_{f*}\circ p_{f*}^{-1}=\phi_{2*}\circ \phi_{1*}^{-1}\circ p_{H*}^{-1} 
\end{equation}
Combining equations (\ref{equation3}) and (\ref{equation4}) we can also obtain the following relation:
\begin{equation}\label{equation6}
q_{H*}=q_{G*}\circ p_{G*}^{-1}\circ \phi_{2*} \circ \phi_{1*}^{-1}.
\end{equation}
To conclude we only need to combine equations equations (\ref{equation5}) and (\ref{equation6}) so as to obtain the desired result.
\begin{equation}
H_{*}=q_{H*}\circ p_{H*}^{-1}=q_{G*}\circ p_{G*}^{-1}\circ \phi_{2*} \circ \phi_{1*}^{-1}\circ p_{H*}^{-1}=q_{G*}\circ p_{G*}^{-1}\circ q_{f*}\circ p_{f*}^{-1}=G_*\circ f_*
\end{equation}
\end{proof}
\end{lem}

Now, using the previous lemmas, we can prove the generalization of Theorem \ref{thm:weakLefschetzFixedPointTheorem}.
\begin{proof}[\textbf{Proof of Theorem \ref{thm:LeftschetzComposiciones}}]
We show the result for the case $F=G_1\circ G_0$, that is to say, $n=1$. For a general $n\in \mathbb{N}$, the proof is a generalization of the case $n=1$, we will indicate how to do it for the case $n=2$ and the general case.  

We have that $G_0:X\multimap Y$ and $G_1:Y\multimap X$ are Vietoris-like multivalued maps. $\Gamma(G_0)$ denotes the graph of $G_0$, where we have the natural projections $p_0:\Gamma(G_0)\rightarrow X$ and $q_{0}:\Gamma(G_0)\rightarrow Y$.  By hypothesis, we know that $p_0$ is a Vietoris-like map. We consider the following composition, $G_1\circ q_0: \Gamma(G_0)\multimap X$, by Lemma \ref{lem:ComposicionNormalYVietorisesVietoris}, $G_1\circ q_0$ is a Vietoris-like multivalued map. Again, $\Gamma(G_1\circ q_0)$ denotes the finite $T_0$ topological space given by the graph of $G_1\circ q_0$, where we also have the natural projections $p_1:\Gamma(G_1\circ q_0)\rightarrow G_0$ and $q_1:\Gamma(G_1\circ q_0)\rightarrow X$. Again, $p_1$ is a Vietoris-like map. We also have that $p_0\circ p_1$ is a Vietoris-like map by Lemma \ref{lem:vietoriscomposicion}. We get the following diagram. 
\[\begin{tikzcd}
\Gamma(G_1 \circ q_0) \arrow{rr}{q_1} \arrow[swap]{d}{p_1} & & X  \\
\Gamma(G_0) \arrow{rr}{q_0} \arrow{rd}{p_0} &  & Y \arrow{u}{G_1} \\
 & X \arrow{ru}{G_0} &
\end{tikzcd}
\]

It is clear that $q_{0*}= G_{0*}\circ p_{0*}$ because $G_{0*}$ is by construction $q_{0*}\circ p_{0*}^{-1}$. By Lemma \ref{lem:CompositionIsWellDefined}, $G_{1*}\circ q_{0*}=(G_1\circ q_0)_*= q_{1*}\circ p_{1*}^{-1}$ and then $G_{1*}\circ q_{0*}\circ p_{1*}= q_{1*}$. From here, we can deduce
\begin{equation}
G_{1*}\circ G_{0*}= G_{1*}\circ q_{0*}\circ p_{0*}^{-1}= q_{1*}\circ p_{1*}^{-1} \circ p_{0*}^{-1}=q_{1*}\circ (p_{0}\circ p_{1})_*^{-1}.
\end{equation}
Therefore, $\Lambda(G_{1*}\circ G_{0*})=\Lambda(q_{1*}\circ (p_{0}\circ p_{1})_*^{-1})\neq 0$, where $q_1,p_0\circ p_1:\Gamma(G_1\circ q_1)\rightarrow X$ are continuous maps between finite $T_0$ topological spaces. Thus, we are in the hypothesis of Theorem \ref{thm:GeneralizacionCoincidenceThm}, so there is a coincidence point for $q_1$ and $p_0\circ p_1$, let us denote that point by $((x,y),z)\in \Gamma(G_1\circ q_0)$, where $(x,y)\in \Gamma(G_0)$, $q_1((x,y),z)=z=p_0( p_1((x,y),z))=p_0(x,y)=x$. Hence, $x=z$ and $x\in G_1( q_0(x,y))=G_1(y)$. But $y\in G_0(x)$, so $x\in G_1( G_0 (x))=F(x)$. 
\\

Now, suppose that $F=G_2\circ G_1\circ G_0$, where $G_0:X\rightarrow Y_1$, $G_1:Y_1\rightarrow Y_2$ and $G_2:Y_2\rightarrow X$ are Vietoris-like multivalued maps and $Y_1,Y_2$ are finite $T_0$ topological spaces. We argue as we did before, we consider $G_1\circ q_0:\Gamma(G_0)\rightarrow Y_2$, which is a Vietoris-like multivalued map. Repeating the arguments used before, it can be obtained that the square and the bottom triangle of the following diagram commute after applying the homological functor. We consider $G_2\circ q_1:\Gamma(G_1\circ q_0)\rightarrow X$, $G_2\circ q_1$ is a Vietoris-like multivalued map by Lemma \ref{lem:ComposicionNormalYVietorisesVietoris}. $\Gamma(G_2\circ q_1)$ denotes the graph of $G_2\circ q_1$, where $p_{2}$ and $q_{2}$ are the projections onto the first and second coordinates respectively.

\[\begin{tikzcd}
  & \Gamma(G_2\circ q_1) \arrow{dl}{p_{2}}  \arrow{drr}{q_{2}} &  &  \\
\Gamma(G_1 \circ q_0) \arrow{rr}{q_1} \arrow[swap]{d}{p_1} & & Y_2 \arrow[swap]{r}{G_2} &  X \\
\Gamma(G_0) \arrow{rr}{q_0} \arrow{rd}{p_0} &  & Y_1 \arrow{u}{G_1} &\\
 & X \arrow{ru}{G_0} &
\end{tikzcd}
\]
By Lemma \ref{lem:vietoriscomposicion}, $p_{0}\circ p_1\circ p_{2}$ is a Vietoris-like map. Using Lemma \ref{lem:CompositionIsWellDefined}, it can be proved that $ \Lambda(G_{2*}\circ G_{1*}\circ G_{0*})=\Lambda(q_{2*}\circ (p_{0}\circ p_1\circ p_{2})_*^{-1})$. By Theorem \ref{thm:GeneralizacionCoincidenceThm}, there exists a coincidence point for $p_{0}\circ p_1\circ p_{2}$ and $q_2$. Let us denote that point by $(((x,y),z),t)\in \Gamma(G_2\circ q_1)$. Then, $p_0(p_1(p_2((((x,y),z),t))=p_0(p_1((x,y),z))=p_0(x,y)=x=t=q_2((((x,y),z),t))$. By construction, $t\in F(x)$, so there exists a fixed point for $F$. 

%

For a general $n$, we only need to use the same arguments described before. Keeping the same notation introduced, we have $0\neq \Lambda(G_{n*}\circ \cdots \circ G_{0*})=\Lambda(q_{n}\circ (p_0\circ \cdots p_n)_*^{-1})$. $p_0\circ \cdots p_n$ is a Vietoris-like map since it is the composition of Vietoris-like maps, Lemma \ref{lem:vietoriscomposicion}. Then $(p_0\circ \cdots \circ  p_n)_*^{-1}$ is well defined. By Lemma \ref{lem:CompositionIsWellDefined}, we get
\begin{align*}
& G_{n*}\circ \cdots \circ G_{1}*\circ G_{0*} = G_{n*}\circ \cdots\circ  G_{1}*\circ q_{0*}\circ p_{0*}^{-1} = G_{n*}\circ \cdots\circ  (G_{1}\circ q_{0})_*\circ p_{0*}^{-1} = \\
& = G_{n*}\circ \cdots\circ G_{2*}\circ  q_{1*}\circ p_{1*}^{-1}\circ p_{0*}^{-1} = G_{n*}\circ \cdots\circ (G_{2}\circ  q_{1})_*\circ p_{1*}^{-1}\circ p_{0*}^{-1} = \\
& = G_{n*}\circ \cdots\circ q_{2*}\circ p_{2*}^{-1}\circ p_{1*}^{-1}\circ p_{0*}^{-1} = \cdots = q_{n*}\circ p_{n*}^{-1}\circ \cdots p_{0*}^{-1}.
\end{align*}
By Theorem \ref{thm:GeneralizacionCoincidenceThm}, there is a coincidence point for $q_n$ and $p_0\circ \cdots p_n$, which implies that there is a fixed point for $F$.

\end{proof}

In general, we cannot expect to obtain every mutlivalued map as a composition of Vietoris-like multivalued maps as we show in the following example.
\begin{ex} Let us consider $X=\{A,B,C\}$ with $A,B<C$ and $F:X\multimap X$ given by $F(A)=B$,$F(B)=A$ and $F(C)=\{A,B \}$. $F$ is clearly a susc multivalued map with images that are not weak homotopy equivalent to a point since $F(C)$ is not connected. Suppose that there exists a sequence of spaces and Vietoris-like multivalued maps such that $F =G_n  \circ \cdots  \circ G_0$, where $G_i:Y_i \multimap Y_{i+1}$ and $Y_0=Y_{n+1}=X$. It is trivial to show that $\Lambda(G_{n*}\circ \cdots G_{0*})\neq 0$ due to the fact that $H_i(X)=0$ for every $i>0$. By Theorem \ref{thm:LeftschetzComposiciones}, we should have a fixed point for $F$ but this is not true. Then, $F$ cannot be expressed as the previous decomposition of multivalued maps.
\end{ex}
\begin{ex}\label{ex:composVietorismuliIsnoVietorisMulti} Let $X$ be the finite $T_0$ topological space given by the Hasse diagram of Figure \ref{fig:composition}. We consider $F:X\multimap X$ defined as follows: \newline
\begin{center}
\begin{tabular}{c  c c c c c }
\hline 
$x$ & $A$ & $B$ & $C$ & $D$ & $E$ \\ 
\hline 
$F(x)$ & $\{A,B,C\}$ & $\{A,B,D\}$ & $\{A,B,C,D\}$ & $\{A,B,C,D\}$ & $\{A,B,C,D\}$ \\ 
\hline 
\end{tabular} 
\end{center}

It is easy to check that $F$ is a susc mutlivalued map, but $F(C)$ is weak homotopy equivalent to $S^1$ because $F(C)$ is indeed a finite model of $S^1$, which implies that $F$ is not a Vietoris-like multivalued map. On the other hand, $F=G_1\circ G_0$, where $G_0:X\multimap X$ and $G_1:X\multimap X$ are susc multivalued maps with contractible images (Vietoris-like multivalued maps) given by
\begin{center}
\begin{tabular}{ c c c c c c }
\hline 
$x$ & $A$ & $B$ & $C$ & $D$ & $E$ \\ 
\hline 
$G_0(x)$ & $\{C\}$ & $\{D\}$ & $\{C,D,B\}$ & $\{C,D,B\}$ & $\{C,D,B\}$ \\ 
\hline 
\end{tabular} 
\end{center}
and \begin{center}
\begin{tabular}{ c c c c c c }
\hline 
$x$ & $A$ & $B$ & $C$ & $D$ & $E$ \\ 
\hline 
$G_1(x)$ & $\{B\}$ & $\{A\}$ & $\{A,B,C\}$ & $\{A,B,D\}$ & $X$ \\ 
\hline 
\end{tabular} 
\end{center}

It is trivial to show that $X$ is contractible, therefore, we can deduce that $\Lambda(G_{1*}\circ G_{0*})\neq 0$. By Theorem \ref{thm:LeftschetzComposiciones}, we know that there exists a fixed point. Looking at $F$ it is easy to verify that $C$ and $D$ are fixed points.
\begin{figure}[h]\center
\includegraphics[scale=1]{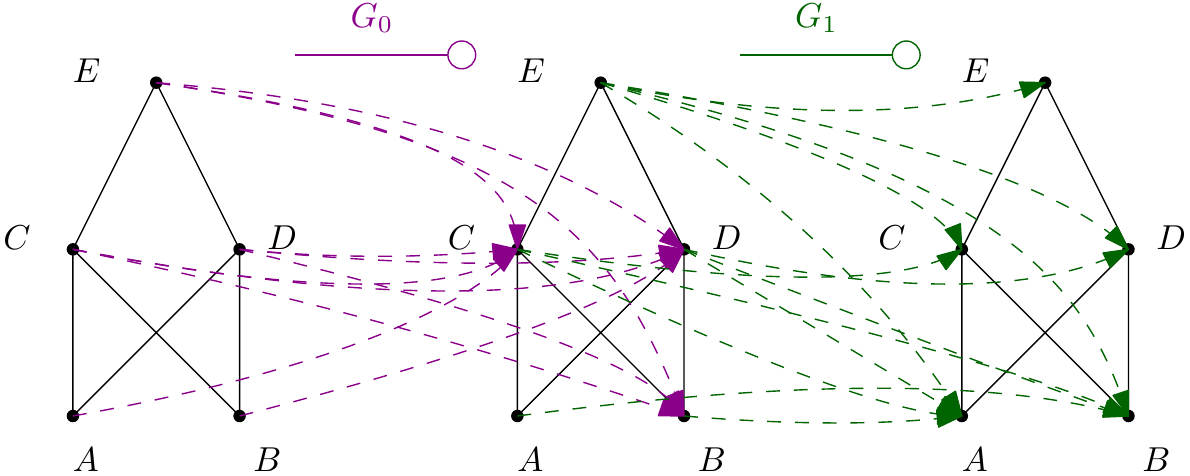}
\caption{Schematic description of $G_0$ and $G_1$ on the Hasse diagram of $X$.}
\label{fig:composition}
\end{figure}
\end{ex}
\begin{rem} In Example \ref{ex:composVietorismuliIsnoVietorisMulti}, we have another example of a composition of Vietoris-like multivalued maps that is not a Vietoris-like multivalued map. In addition, it can be proved that the graph of $F$ is not contractible, it has the same weak homotopy type of a circle.
\end{rem}

\section{Approximation of dynamical systems}\label{sec:application}

\begin{df}\cite{hardie1993homotopy,walker1981homotopy} Given a finite $T_0$ topological space $X$, the finite barycentric subdivision of $X$ is given by $\mathcal{X}(\mathcal{K}(X))$.
\end{df}
The finite barycentric subdivision of a finite $T_0$ topological space $X$ will be denoted by $X'$. If it is applied $n$-times, the $n$-th finite barycentric subdivision of $X$ will be denoted by $X^n$. If $K$ is a simplicial complex we will also keep the same notation for its barycentric subdivisions, i.e., the $n$-th barycentric subdivision of $K$ is denoted by $K^n$.
\begin{rem} Let $X$ be a finite $T_0$ topological space, $X'$ can be seen as the poset whose points are the chains in $X$, i.e., if $x'\in X'$, $x'$ can be seen as a chain $x_1<...<x_n$ in $X$. Hence, the partial order of $X'$ is given by the subset relation.
\end{rem}
There is a natural map between a finite $T_0$ topological space  $X$ and its finite barycentric subdivision $X'$, $h:X'\rightarrow X$ is given by $h(x')=x_n$, where $x'$ is a chain $x_1<...<x_n$ of $X$. It is easy to get that $h$ is continuous. In fact, it can be obtained the following.
\begin{prop}[\cite{may1966finite}] If $X$ is a finite $T_0$ topological space, then $h:X'\rightarrow X$ is a weak homotopy equivalence 
\end{prop} 
\begin{rem}\label{rem:pInduceIdentidad} It is easy to show that $|\mathcal{K}(h)|$ is simplicially close to the identity map on $|\mathcal{K}(X)|$. Then, $h:X'\rightarrow X$ induces the identity in homology.
\end{rem}
More details about the previous results can be found in \cite{may1966finite}.

\begin{prop}\label{prop:pisVietorislikemultimap} Let $X$ be a finite $T_0$ topological space. Then, $h:X'\rightarrow X$ is a Vietoris-like map.
\begin{proof}
We take a chain $x_1<...<x_n$ in $X$, we denote $A=\bigcup_{i=1}^n h^{-1}(x_i)$. We define $f:A\rightarrow A$ given by $f(y)=y\cap \{x_j\}_{j=1}^n$, where we denote by $y\cap \{x_j\}_{j=1}^n$ the subchain of $y$ given just by elements of $\{x_j\}_{j=1}^n$. We prove the continuity of $f$, if $y\leq z$, $y$ is a subchain of $z$, by the construction of $f$, we get easily that $f(y)\leq f(z)$. In addition, $f:A\rightarrow f(A)$ is a retraction. We also have that $f\leq id$. If $y\in A$, $f(y)$ is a subchain of $y$, so $f(y)\leq y=id(y)$. From here, it is easy to deduce that $f(A)\subset A$ is a strong deformation retract of $A$. On the other hand, $f(A)$ contains a maximum, which is  $x_1<x_2<...<x_{n-1}<x_n$. Thus, $f(A)$ is contractible. 
\end{proof}
\end{prop}
Given a finite $T_0$ topological space $X$, we denote by $h_{n,n+1}:X^{n+1}\rightarrow X^n$ the natural weak homotopy equivalence described before. If $m\geq n$, where $n,m\in \mathbb{N}$, $h_{n,m}: X^m\rightarrow X^n$ is given by $h_{n,n+1}\circ \cdots\circ h_{m-2,m-1}\circ h_{m-1,m}$.

\begin{ex}\label{ex:cutreWeakaunpunto} Let $X$ be a finite $T_0$ topological space that has the weak homotopy type of a point. Let us denote $X^0=X$. By Proposition \ref{prop:pisVietorislikemultimap}, Lemma \ref{lem:vietoriscomposicion} and Corollary \ref{cor:CoincidenciaMultivaluadaYnormal}, for every Vietoris-like multivalued map $F:X^m\multimap X^n$ and $m>n$, there exists a point $x\in X^m$ such that $h_{n,m}(x)\in F(x)$.
\end{ex}

We consider the multivalued map $H:X\multimap X'$ given by $H(x)=h^{-1}(x)$. It is important to observe that $H(x)$ consists of chains containing $x$ as a maximum element. More generally, we can consider $H_{n,m}:X^n\multimap X^m$ given by $H_{n,m}(x)=h_{n,m}^{-1}(x)$, for every $m\geq n$ and $m,n\in \mathbb{N}$

On the other hand, $H$ is an example of a usc mutlivalued map such that $H(x)$ contains a minimum for every $x\in X$ and is a Vietoris-like multivalued map. The chain that consists of a single element $x\in X$ is the minimum of $H(x)$ for every $x\in X$. Now, we prove that $H$ is usc. We consider $x,y\in X$ such that $x\leq y$. If $x'\in H(x)$,  then $x'$ is a chain of $X$ containing $x$ as a maximum element. We can extend this chain to a new chain containing $y$ as a maximum element since $x\leq y$. But, the new chain is an element of $H(y)$. Therefore, $H$ is a usc mutlivalued map, Definition \ref{def:semicontinuous}.

\begin{prop}\label{prop:PisVietorislikemultimap} Let $X$ be a finite $T_0$ topological space and  $n,m\in \mathbb{N}$ are such that $m\geq n$. $H_{n,m}:X^n\multimap X^m$ is a Vietoris-like multivalued map.
\begin{proof}
By Lemma \ref{lem:vietoriscomposicion} and Proposition \ref{prop:pisVietorislikemultimap}, $h_{n,m}$ is a Vietoris-like map. Finally, by Proposition \ref{prop:vietorislikempainducenmultivevaluadasvietoris}, we obtain that $H_{n,m}$ is a Vietoris-like multivalued map.
\end{proof}
\end{prop}
\begin{rem}\label{rem:PmultiInduceIdentidad} It is easy to show that $H$ induces the identity in homology, it is an immediate consequence of Remark \ref{rem:pInduceIdentidad} and the commutativity of the following diagram, where $p$ and $q$ denote the projection onto the first and second coordinates respectively from the graph of $H$.
\[\begin{tikzcd}
& \Gamma(H) \arrow{dr}{q} \arrow{dl}{p}  & \\
X & & X' \arrow{ll}{h}
\end{tikzcd}
\]
It is clear that $q$ is a weak homotopy equivalencce by the 2-out-of-3 property, but it can be obtained that $q$ is indeed a Vietoris-like map, Proposition \ref{prop:secondprojectionVietoris}.
\end{rem}

\begin{df} An inverse sequence in the topological category consists of a topological space $X_n$ for every $n\in \mathbb{N}$ and a continuous map $p_{n,m}:X_{m}\rightarrow X_n$ for each pair $m\geq n$. Moreover, it is required that $p_{n,n}$ is the identity map and that $m\geq n $ and $n\geq s$ implies $p_{s,m}=p_{n,m}\circ p_{s,n}$. An inverse sequence is usually denoted by $(X_n,p_{n,n+1})$, where the continuous maps are called the bonding maps and the topological spaces are called the terms. 
\end{df}
By definition, for an inverse sequence $(X_n,p_{n,n+1})$, it is enough to know $p_{n,n+1}$ for every $n\in \mathbb{N}$ because the other bonding maps are obtained just by compositions. We recall the definition of inverse limit for the topological category.
\begin{df}\label{ex:composicionmultivietorisnoesvietoris} The inverse limit $X$ of an inverse sequence $(X_n,p_{n,n+1})$ in the topological category is the subspace of $\Pi_{i\in \mathbb{N}} X_i $, which consists of all points $x$ satisfying $$\pi_{n}(x)=p_{n,m}(\pi_{m}(x)),\quad \quad n\leq m.$$ 
Where $\pi_n:\Pi_{i\in \mathbb{N}} X_i\rightarrow X_n$ denotes the natural projection. 
\end{df}
For a complete exposition of the notion of inverse limit and inverse sequence, see for instance \cite{mardevsic1982shape}.

In \cite{clader2009inverse}, given a simplicial complex $K$, it is obtained an inverse sequence $(X^n,h_{n,n+1})$ of finite $T_0$ topological spaces, where $X^n$ is the $n$-th finite barycentric subdivision of $\mathcal{X}(K)$ and $\mathcal{X}(K)$ is considered with the opposite order or opposite topology, and $h_{n,n+1}:X^{n+1}\rightarrow X^{n}$ is the natural weak homotopy equivalence defined before. The inverse limit of this inverse sequence contains a homeomorphic copy of $K$ which is a strong deformation retract of the inverse limit. Then, the inverse limit of this inverse sequence reconstructs the homotopy type of $K$.

If $f:|K|\rightarrow |L|$ is a continuous map between the geometric realizations of simplicial complexes, then there is a natural morphism induced by $f$ over the inverse sequences related to $K$ and $L$. $(X^n,h_{n,n+1})$ denotes the inverse sequence of finite $T_0$ topological spaces associated to $K$ and $(Y^n,h_{n,n+1})$ denotes the inverse sequence of finite $T_0$ topological spaces associated to $Y$, where we are denoting the bonding maps of the two inverses sequences with the same notation for simplicity. By the simplicial approximation theorem, there exists $n_0\in \mathbb{N}$ such that for every $m\geq n_0$ there exists a simplicial map from the $m$-th barycentric subdivision of $K$ to $L$ which is simplicially close to $f$ in $|L|$, that is to say, there exists $f(0)=n_0$ and a simplicial map $f_0:K^{f(0)}\rightarrow L$ simplicially close to $f$. Using Theorem \ref{thm:McCordX}, we can obtain the finite version of the previous result, i.e., $\mathcal{X}(f_0):X^{f(0)}\rightarrow Y^0$. Now, we can consider the barycentric subdivision of $L$, denoted by $L^1$, and repeat the same arguments so as to obtain $f(1)\geq f(0)$, $f_1:K^{f(1)}\rightarrow L^1$ and $\mathcal{X}(f_1):X^{f(1)}\rightarrow Y^1$. We can follow inductively until we get an increasing map $f:\mathbb{N}\rightarrow \mathbb{N}$ and $\mathcal{X}(f_n):X^{f(n)}\rightarrow Y^n$ continuous for every $n\in \mathbb{N}$. $\mathcal{X}(f_n)$ will be also denoted by $f_n$ for simplicity. Moreover, the following diagram is commutative after applying a homological functor for every $m\geq n$.
\[\begin{tikzcd}[cramped,sep=large] 
X^{f(n)} \arrow{r}{f_n} & Y^n \\
X^{f(m)} \arrow{u}{h_{f(n),f(m)}} \arrow{r}{f_m} & Y^m \arrow{u}{h_{n,m}}
\end{tikzcd}
\]
Hence, every continuous map between the geometric realization of simplicial complexes induces a morphism between the inverse sequences associated to them.

If $f:|K|\rightarrow |K|$ is a continuous map between the geometric realization of a simplicial complex $K$, we have an inverse sequence $(X^n,h_{n,n+1})$ and a morphism $(f_n)$ from that inverse sequence to itself, we only need to repeat the previous arguments. We can construct a new inverse sequence using as bonding maps continuous maps of $(f_n )$. We start with $f_0:X^{f(0)}\rightarrow X^0$, we rename $f_{0,1}=f_0$ and $X^1=X^{f(0)}$. We have $f(1)>f(0)$, so we can consider $f_{1,2}=f_{f(f(0))}:X^{f(f(0))}\rightarrow X^{1}$, we rename $X^{f(f(0))}$ just by $X^2$ and continue this process. Then, we obtain a new inverse sequence $(X^n,f_{n,n+1})$, we also have an inverse sequence isomorphic to $(X^n,h_{n,n+1})$ with the same terms of $(X^n,f_{n,n+1})$ and bonding maps of $(X^n,h_{n,n+1})$ just relabelling and using cofinality. This new inverse sequence will also be denoted by  $(X^n,h_{n,n+1})$. It is also trivial to check that $(X^n,h_{n,n+1})$ preserves the same good properties of reconstruction in its inverse limit. An inverse sequence obtained from a geometric realization $|K|$ of a simplicial complex $K$ as we did before will be called finite approximative sequence for $f$. Now, we can compare all the bonding maps in a direct way because for every $m>n$ we have $f_{n,m},h_{n,m}:X^m\rightarrow X^n$, where $h_{n,m}$ is a Vietoris-like map.
We define $\Lambda_{n,m}(f)$ as $\Lambda(f_{n,m*}\circ h_{n,m*}^{-1})$. By Theorem \ref{thm:GeneralizacionCoincidenceThm}, if $\Lambda_{n,m}(f)\neq 0$, there exists a coincidence point for $f_{n,m} $ and $h_{n,m}$. But, $|\mathcal{K}(h_{n,m})|$ is homotopic to the identity map and $|\mathcal{K}(f_{n,m})|$ is homotopic to $f$. By conjugacy, it is easy to deduce that $\Lambda_{n,m}(f)=\Lambda(f_{n,m*}\circ h_{n,m*}^{-1})=\Lambda(\mathcal{K}(f_{n,m})_*\circ \mathcal{K}(h_{n,m})^{-1}_*)=\Lambda(f)$.

We can define a multivalued map for each level of the inverse sequence, we only need to consider the multivalued map $H_{n,m}$ induced by $h_{n,m}$, where $m\geq n$, that is to say, $H_{n,m}(x)=h_{n,m}^{-1}(x)$ for every $x\in X^n$. Thus, $F_{n+1}:X^{n+1}\multimap X^{n+1}$ is given by $F_{n+1}=H_{n,n+1}\circ f_{n,n+1}$ and we have the following diagram.

\[\begin{tikzcd}[cramped,sep=large] 
X^0  & X^1 \arrow[d,dash,"F_1"] \arrow{l}{h_{0,1}} &  X^2 \arrow{l}{h_{1,2}} \arrow[d,dash,"F_2"] & \arrow{l}{h_{2,3}} \dots & X^n \arrow{l}{h_{n-1,n}} \arrow[d,dash,"F_n"] & X^{n+1} \arrow{l}{h_{n,n+1}} \arrow[d,dash,"F_{n+1}"]&  X^{n+2} \arrow{l}{h_{n+1,n+2}}  \arrow[d,dash,"F_{n+2}"]& \arrow{l}{h_{n+2,n+3}}  \\
X^0  & X^1 \arrow{l}{f_{0,1}} &  X^2 \arrow{l}{f_{1,2}} & \arrow{l}{f_{2,3}} \dots & X^n \arrow{l}{f_{n-1,n}} & X^{n+1} \arrow{l}{f_{n,n+1}} &  X^{n+2} \arrow{l}{f_{n+1,n+2}} & \arrow{l}{f_{n+2,n+3}} 
\end{tikzcd}
\]

\begin{prop}\label{prop:FesVietoris} $F_{n+1}:X^{n+1}\multimap X^{n+1}$ is a Vietoris-like multivalued map such that $F_{n+1*}=H_{n,n+1*}\circ f_{n+1,n*}$ for every $n\in \mathbb{N}$.
\end{prop} 
\begin{proof}
$H_{n,n+1}$ is a Vietoris-like multivalued map, Proposition \ref{prop:PisVietorislikemultimap}. By Lemma \ref{lem:ComposicionNormalYVietorisesVietoris}, $F_{n+1}$ is a Vietoris-like multivalued map. The last part is an immediate consequence of Lemma \ref{lem:CompositionIsWellDefined}.
\end{proof}
By the theory developed in previous sections, it is immediate to get the following proposition.
\begin{prop}\label{prop:FixedPointFn+1} If $\Lambda(F_{n+1*})\neq 0$, there exists a point $x\in X_{n+1}$ such that $x\in F_{n+1}(x)$, where $n\in \mathbb{N}$.
\end{prop}

\begin{rem}\label{rem:fijoigualcoincidencia} By construction, it can be deduced that $x\in X^{n+1}$ is a fixed point for $F_{n+1}$ if and only if $x$ is a coincidence point for $h_{n,n+1}$ and $f_{n,n+1}$.
\end{rem}

\begin{cor} If $\Lambda(f)\neq 0$, there exists a point $x_{n+1}\in X^{n+1}$ such that $x_{n+1}\in F_{n+1}(x_{n+1})$, for every $n\in \mathbb{N}$.
\begin{proof}
By Remark \ref{rem:PmultiInduceIdentidad}, it is easy to deduce $\Lambda(f)=\Lambda(f_{n,n+1})=\Lambda(f_{n,n+1*}\circ h_{n,n+1*}^{-1})=\Lambda(H_{n,n+1*}\circ f_{n,n+1*})=\Lambda(F_{n+1*})$. From here, we get the desired result.
\end{proof}
\end{cor}


\begin{thm} If $f:|K|\rightarrow |K|$ is a continuous map, where $K$ is a simplicial complex, then $f$ has a fixed point if and only if there exist a finite approximative sequence for $f$, $(X^n,h_{n,n+1})$, a sequence $\{x_{n+1}\}_{n\in \mathbb{N}}$ and $m\in \mathbb{N}$ such that $x_{n+1}\in X^{n+1}$, $x_{n}=h_{n,n+1}(x_{n+1})$ for every $n\in \mathbb{N}$ and $x_{n+1}\in F_{n+1}(x_{n+1})$ for every $n+1\geq m$.
\begin{proof}
Firstly, we suppose that there exist a finite approximative sequence for $f$, $(X^n,h_{n,n+1})$, and a sequence $\{x_{n+1}\}$ such that $x_{n+1}\in X_{n+1}$, $x_{n}=h_{n,n+1}(x_{n+1})$ for every $n\in \mathbb{N}$ and $x_{n+1}\in F_{n+1}(x_{n+1})$ for every $n+1\geq m$.  By Remark \ref{rem:fijoigualcoincidencia}, $h_{n,n+1}(x_{n+1})=f_{n,n+1}(x_{n+1})$ for every $n\in \mathbb{N}$. We have that $|\mathcal{K}(h_{n,n+1})|$ is simplicially close to the identity map. We also know that $|\mathcal{K}(f_{n,n+1})|$ is simplicially close to $|f_{n,n+1}|:|K^{n+1}|\rightarrow |K^{n}|$ and $|f_{n,n+1}|$ is simplicially close to $|f|$. The maximum diameter for a closed simplex in $|K^n|$ is denoted by $\epsilon_n$ for every $n\in \mathbb{N}$. If $x_{n+1}$ is viewed as a point of $|\mathcal{K}(X_{n+1})|=|K^{n+2}|=|K|$, then by the triangle inequality, $d(x_{n+1},f(x_{n+1}))$ $\leq d(x_{n+1},|\mathcal{K}(h_{n,n+1})|(x_{n+1}))$ $+$ $d(|\mathcal{K}(h_{n,n+1})(x_{n+1})|,|f_{n,n+1}|(x_{n+1}))$ $+$ $d(|f_{n,n+1}|(x_{n+1}),f(x_{n+1}))$ $\leq 3\epsilon_n$, where we also know that $lim_{n\rightarrow\infty }3\epsilon_n=0$, this is due to the fact that after a barycentric subdivision the diameters of the new simplices are smaller, see for example \cite{hatcher2000algebraic,spanier1981algebraic}. In addition, every $x_{n+1}\in X^{n+1}$ can be seen a closed simplex of $|K|$ since $X^{n+1}$ is the face poset of $K^{n+1}$, where have that $x_{n+1}\subset x_n$ for every $n\in \mathbb{N}$. Hence, $\cap_{n\in \mathbb{N}}x_{n+1}$ is non-empty because it is the intersection of a nested sequence of compact sets. In fact, the diameter of $\cap_{n\in \mathbb{N}}x_n$ is zero, which implies that it is a point $x*$. We show that $f(x*)=x*$. The sequence $\{ x_{n+1}\}_{n\in \mathbb{N}}$, where we are treating now $x_n$ as a point of $|K|$, is convergent to $x*$. From here, we can deduce the desired result.

If $f$ has a fixed point denoted by $t$, we only need to consider a triangulation that has $t$ as a vertex. Then, we can construct the inverse sequence satisfying the desired property.

\end{proof}

\end{thm}
\begin{rem}  The idea of this construction is that the dynamics generated by $F_{n+1}$ approximates the classical dynamics generated by $f$. As long as $n$ is bigger, the approximation to $f$ is better. 
\end{rem}

\bibliography{bibliografia}

\begin{thebibliography}{10}

\bibitem{alexandroff1937diskrete}
{\sc Alexandroff, P.~S.}
\newblock Diskrete {R}äume.
\newblock {\em Mathematiceskii Sbornik (N.S.) 2}, 3 (1937), 501--519.

\bibitem{baclawski1979fixed}
{\sc Baclawski, K., and Björner, A.}
\newblock Fixed points in partially ordered sets.
\newblock {\em Advances in Mathematics 31}, 3 (1979), 263--287.

\bibitem{barmak2011algebraic}
{\sc Barmak, J.~A.}
\newblock {\em Algebraic topology of finite topological spaces and
  applications}, vol.~2032.
\newblock Springer, 2011.

\bibitem{barmak2011onQuillen}
{\sc Barmak, J.~A.}
\newblock On {Q}uillen’s {T}heorem {A} for posets.
\newblock {\em Journal of Combinatorial Theory, Series A 118\/} (2011),
  2445–--2453.

\bibitem{barmak2007minimal}
{\sc Barmak, J.~A., and Minian, E.~G.}
\newblock Minimal finite models.
\newblock {\em Journal of Homotopy and Related Structures 2}, 1 (2007),
  127--140.

\bibitem{barmak2020Lefschetz}
{\sc Barmak, J.~A., Mrozek, M., and Wanner, T.}
\newblock A {L}efschetz fixed point theorem for multivalued maps of finite
  spaces.
\newblock {\em Mathematische Zeitschrift 294\/} (2020), 1477–--1497.

\bibitem{begle1950vietoris}
{\sc Begle, E.~G.}
\newblock The {V}ietoris {M}apping {T}heorem for {B}icompact {S}paces.
\newblock {\em Annals of Mathematics 51}, 3 (1950), 524--543.

\bibitem{chocano2020topological}
{\sc Chocano, P.~J., Mor\'on, M.~A., and Ruiz~del Portal, F.~R.}
\newblock Topological realizations of groups in {A}lexandroff spaces.

\bibitem{clader2009inverse}
{\sc Clader, E.}
\newblock Inverse limits of finite topological spaces.
\newblock {\em Homology, Homotopy and Applications 11}, 2 (2009), 223--227.

\bibitem{eilenberg1946fixed}
{\sc Eilenberg, S., and Montgomery, D.}
\newblock Fixed point {T}heorems for {M}ulti-{V}alued {T}ransformations.
\newblock {\em American Journal of Mathematics 68}, 2 (1946), 214--222.

\bibitem{forma1998combinatorial}
{\sc Forman, R.}
\newblock Combinatorial vector fields and dynamical systems.
\newblock {\em Mathematische Zeitfreit 228}, 4 (1998), 629--681.

\bibitem{forma1998morse}
{\sc Forman, R.}
\newblock Morse theory for cell complexes.
\newblock {\em Advances in Mathemathics 134}, 1 (1998), 90--145.

\bibitem{hardie1993homotopy}
{\sc Hardie, K.~A., and Vermeulen, J. J.~C.}
\newblock Homotopy theory of finite and locally finite ${T}_0$-spaces.
\newblock {\em Exposition Math}, 11 (1993), 331--341.

\bibitem{hatcher2000algebraic}
{\sc Hatcher, A.}
\newblock {\em Algebraic topology}.
\newblock Cambridge Univ. Press, Cambridge, 2000.

\bibitem{lipinski2020conley}
{\sc Lipi\'nski, M., Kubika, J., Mrozek, M., and Wanner, T.}
\newblock Conley-{M}orse-{F}orman theory for generalized combinatorial
  multivector fields on finite topological spaces.

\bibitem{mardevsic1982shape}
{\sc Marde{\v{s}}ic, S., and Segal, J.}
\newblock {\em Shape theory: the inverse system approach}.
\newblock North-Holland Mathematical Library, 1982.

\bibitem{may1966finite}
{\sc May, J.~P.}
\newblock Finite spaces and larger contexts.
\newblock {\em Unpublished book\/} (2016).

\bibitem{mccord1966singular}
{\sc McCord, M.~C.}
\newblock Singular homology groups and homotopy groups of finite topological
  spaces.
\newblock {\em Duke Mathematical Journal 33}, 3 (1966), 465--474.

\bibitem{mischaikow2013morse}
{\sc Mischaikow, K., and Nanda, V.}
\newblock Morse {T}heory for {F}iltrations and {E}fficient {C}omputations of
  {P}ersistent {H}omology.
\newblock {\em Discrete and Computational Geometry 50\/} (2013), 330--353.

\bibitem{mondejar2020reconstruction}
{\sc Mondéjar, D., and Morón, M.~A.}
\newblock Reconstruction of compacta by finite approximation and inverse
  persistence.
\newblock {\em Revista Matemática Complutense}, in Press (2020),
  https://doi.org/10.1007/s13163--020--00356--w.

\bibitem{mrozek2017conley}
{\sc Mrozek, M.}
\newblock Conley–{M}orse–{F}orman {T}heory for {C}ombinatorial
  {M}ultivector {F}ields on {L}efschetz {C}omplexes.
\newblock {\em Foundations of Computational Mathematics volume}, 17 (2017),
  1585–1633.

\bibitem{munkres1984elements}
{\sc Munkres, J.~R.}
\newblock {\em Elements of {A}lgebraic topology}.
\newblock Addison--Wesley, 1984.

\bibitem{powers1970multi}
{\sc Powers, M.~J.}
\newblock Multi-valued mappings and {L}efschetz fixed point theorems.
\newblock {\em Mathematical Proceedings of the Cambridge Philosophical Society
  68}, 3 (1970), 619--630.

\bibitem{rival1976afixed}
{\sc Rival, I.}
\newblock A fixed point theorem for finite partially ordered sets.
\newblock {\em J. Combin. Theory A 21\/} (1976), 309--318.

\bibitem{spanier1981algebraic}
{\sc Spanier, E.~H.}
\newblock {\em Algebraic topology}.
\newblock Springer-Verlag, New York-Berlin, 1981.

\bibitem{vietoris1927uber}
{\sc Vietoris, L.}
\newblock Über den höheren {Z}usammenhang kompakter {R}äume und eine
  {K}lasse von zusammenhangstreuen abbildungen.
\newblock {\em Matematische Annalen 97\/} (1927), 454--472.

\bibitem{walker1981homotopy}
{\sc Walker, J.~W.}
\newblock Homotopy type and {E}uler characteristics of partially ordered sets.
\newblock {\em Europ. J. Combinatorics 2}, 2 (1981), 373--384.

\end{thebibliography}
\bibliographystyle{acm}

\newcommand{\Addresses}{{
  \bigskip
  \footnotesize

  \textsc{ P.J. Chocano, Departamento de \'Algebra, Geometr\'ia y Topolog\'ia, Universidad Complutense de Madrid, Plaza de Ciencias 3, 28040 Madrid, Spain}\par\nopagebreak
  \textit{E-mail address}:\texttt{pedrocho@ucm.es}

  \medskip

\textsc{ M. A. Mor\'on,  Departamento de \'Algebra, Geometr\'ia y Topolog\'ia, Universidad Complutense de Madrid and Instituto de
Matematica Interdisciplinar, Plaza de Ciencias 3, 28040 Madrid, Spain}\par\nopagebreak
  \textit{E-mail address}: \texttt{ma\_moron@mat.ucm.es}

  \medskip

\textsc{ F. R. Ruiz del Portal,  Departamento de \'Algebra, Geometr\'ia y Topolog\'ia, Universidad Complutense de Madrid and Instituto de
Matematica Interdisciplinar
, Plaza de Ciencias 3, 28040 Madrid, Spain}\par\nopagebreak
  \textit{E-mail address}: \texttt{R\_Portal@mat.ucm.es}

}}

\Addresses

\end{document}